\documentclass[11pt]{amsart}

\usepackage[a4paper,hmargin=3.5cm,vmargin=4cm]{geometry}
\usepackage{amsfonts,amssymb,amscd,amstext}
\usepackage{graphicx}
\usepackage[dvips]{epsfig}
\usepackage{color}

\usepackage{fancyhdr}
\usepackage{verbatim}
\usepackage[utf8]{inputenc}
\usepackage{hyperref}

\input xy
\xyoption{all}

\usepackage{times}
\usepackage{enumerate}
\usepackage{titlesec}
\usepackage{mathrsfs}

\titleformat{\section}
{\filcenter\bfseries\large} {\thesection{.}}{0.2cm}{}
\titleformat{\subsection}[runin]
{\bfseries} {\thesubsection{.}}{0.15cm}{}[.]
\titleformat{\subsubsection}[runin]
{\em}{\thesubsubsection{.}}{0.15cm}{}[.]
\usepackage[up,bf]{caption}

\newtheorem{theorem}{Theorem}[section]
\newtheorem{proposition}[theorem]{Proposition}
\newtheorem{claim}[theorem]{Claim}
\newtheorem{lemma}[theorem]{Lemma}

\newtheorem{conjecture}[theorem]{Conjecture}

\theoremstyle{definition}
\newtheorem{definition}[theorem]{Definition}

\numberwithin{equation}{section}
\numberwithin{figure}{section}

\newcommand\Rcal{\mathcal{R}}

\newcommand\C{\mathbb{C}}
\newcommand\D{\overline{\mathbb D}}

\renewcommand\D{\mathbb D}

\newcommand\R{\mathbb{R}}

\newcommand\igot{\mathfrak{i}}

\renewcommand\igot{\mathfrak{i}}

\renewcommand\imath{\igot}

\newcommand\di{\partial}

\begin{document}

\fancyhead[LO]{Curvature of minimal graphs}
\fancyhead[RE]{ D. Kalaj, P. Melentijevi\'c}
\fancyhead[RO,LE]{\thepage}

\thispagestyle{empty}

\vspace*{1cm}
\begin{center}
{\bf\LARGE   Gaussian curvature conjecture for  minimal graphs}

\vspace*{0.5cm}

{\large\bf  David Kalaj \quad  Petar Melentijevi\'c}
\end{center}


\vspace*{1cm}

\begin{quote}

{\small
\noindent {\bf Abstract}\hspace*{0.1cm}
In this paper, we solve the longstanding Gaussian curvature conjecture of a minimal graph $S$ over the unit disk. The conjecture asserts that for any minimal graph above the unit disk, the Gaussian curvature at the point directly above the origin satisfies the sharp inequality \( |\mathcal{K}| < \frac{\pi^2}{2} \). We first reduce the conjecture to the problem of estimating the Gaussian curvature of certain Scherk-type minimal surfaces defined over bicentric quadrilaterals inscribed in the unit disk, containing the origin. We then provide a sharp estimate for the Gaussian curvature of these minimal surfaces at the point above the origin. Our proof employs complex-analytic methods, as the minimal surfaces in question allow a conformal harmonic parameterization.

\vspace*{0.2cm}

\noindent{\bf Keywords}\hspace*{0.1cm} Conformal minimal surface, minimal graph, curvature}

\vspace*{0.2cm}


\vspace*{0.1cm}

\noindent{\bf MSC (2010):}\hspace*{0.1cm} 53A10, 32B15, 32E30, 32H02

\vspace*{0.1cm}
\noindent{\bf Date: \today} 
\end{quote}

\vspace{0.2cm}

\section{Introduction and the main result}

Let $D(w_0,R)$ be an open disk in $\mathbf{R}^2$ centered at $w_0$ with radius $R$  and let \( \mathbf{f}: D(w_0, R) \to \mathbb{R} \) be a \( C^2 \)-smooth function that satisfies the minimal surface equation:
\[
\mathbf{f}_{uu}(1 + \mathbf{f}_v^2) - 2 \mathbf{f}_u \mathbf{f}_v \mathbf{f}_{uv} + \mathbf{f}_{vv}(1 + \mathbf{f}_u^2) = 0.
\]

The graph of \( \mathbf{f} \), denoted by \( S = \mathrm{Graph}_{\mathbf{f}} = \{(u, v, \mathbf{f}(u, v))\} \), is a minimal graph in \( \mathbb{R}^3 \) over the disk \( D(w_0, R) \). The Gaussian curvature \( \mathcal{K}(P) \) of the graph \( S \) at a point \( P = (u, v, \mathbf{f}(u, v)) \) is given by the formula:
\[
\mathcal{K}(P) = \frac{\mathbf{f}_{uu} \mathbf{f}_{vv} - \mathbf{f}_{uv}^2}{(1 + \mathbf{f}_u^2 + \mathbf{f}_v^2)^2}.
\]

Now, let \( \xi \) be a point above \( w_0 \). A longstanding open problem in the theory of minimal surfaces is to determine the precise value of the constant \( c_0 \) in the inequality
\begin{equation}\label{co}
|\mathcal{K}(\xi)| \le \frac{c_0}{R^2}.
\end{equation}
This problem traces its origins back to 1952, as established by E. Heinz \cite{HEINZ}. More specifically, Heinz proved the famous Bernstein theorem, which asserts the existence of a positive constant \( c_0 \) such that the inequality holds. This result is further explored in the monograph by T. Colding and W. Minicozzi II \cite[Theorem~2.3]{COLDING}. For a generalization to higher dimensions, regarding minimal hypersurfaces in \( \mathbb{R}^{n+1} \) for \( n \leq 5 \), see the paper by Schoen, Simon, and Yau \cite{acta}.

Since then, numerous improvements and generalizations have been made by various authors, including E. Hopf \cite{HOPF}, R. Finn and R. Osserman \cite{FINNOSSERMAN}, J. C. C. Nitsche \cite{NITSCHE2}, and R. Hall \cite{HALL1, HALL2}. Notably, in 1953, E. Hopf \cite{HOPF} raised the question of the optimality of the constant \( c_0 \) in the Heinz inequality \eqref{co}, and proposed the following conjecture:

\begin{conjecture}
\begin{equation}\label{conject}
c_0=\frac{\pi^2}{2}.\end{equation}
\end{conjecture}  This  conjecture has been proved by Finn and Osserman  \cite{FINNOSSERMAN} under the additional assumption that the tangential plane of the minimal surface is horizontal at the point $\xi$. Further J. C. C. Nitsche \cite{NITSCHE2} proved \eqref{conject} for symmetric minimal surfaces.

The conjecture has also been  mentioned  by Duren in \cite[p.~185]{DUREN} and Finch in \cite[p.~401]{FINCH}.

In the paper, we solve the Conjecture 1.1 by proving the following theorem:

\begin{theorem}\label{mainres}
Assume that $S$ is a minimal graph over the disk $D(w_0,R)$, where $\xi$ is a point above $w_0$. Then the sharp inequality
\begin{equation}\label{conj3}|\mathcal{K}(\xi)|<\frac{\pi^2}{2}\frac{1}{R^2}\end{equation}
holds. The equality in \eqref{conj3} cannot be attained.
\end{theorem}
We build on the original idea of Finn and Osserman \cite{FINNOSSERMAN}, using a curvature comparison argument between a minimal graph and its corresponding Scherk-type minimal graph. In his Lectures on Minimal Surfaces, Nitsche introduces generalized Scherk surfaces as minimal graphs with simple poles positioned at four points on the unit circle, corresponding to angles $\alpha_k$ ($k=1,2,3,4$) \cite[Sec.~2.4]{nitlectures}. We also make use of these surfaces, which generalize classical Scherk surfaces and naturally arise as minimal graphs over quadrilateral domains.
The first step toward proving the conjecture is the following result (Theorem~\ref{prejprej}), a version of which was previously announced by the first author in the unpublished manuscript~\cite{KALAJ}, as part of an approach to obtain partial progress on the conjecture. For completeness and to keep the manuscript self-contained, we provide full proof here.
To formulate Theorem~\ref{prejprej}, we begin with the following:
\begin{definition}
In Euclidean geometry, a bicentric quadrilateral is a convex quadrilateral that has both an incircle (a circle tangent to all four sides) and a circumcircle (a circle passing through all four vertices).
Assume that  $Q=Q(a,b,c,d)$  is a bicentric quadrilateral inscribed in the unit disk $\mathbb{D}$. A minimal graph  $S=\{(u,v,\mathbf{f}(u,v)), (u,v)\in Q\}$ over the quadrilateral $Q$ is called a Scherk type surface if it satisfies $\mathbf{f}(u,v)\rightarrow + \infty$ when $(u,v) \rightarrow \zeta\in(a,b)\cup (c,d)$ and $\mathbf{f}(u,v)\rightarrow -\infty$ when $(u,v) \rightarrow \zeta\in(b,c)\cup (a,d)$ .
\end{definition}
\begin{theorem}\label{prejprej}
	For every $w\in\D$, there exist four different points $a_0, a_1,a_2,a_3\in\mathbb{T}=\partial\mathbb{D}$, such that the quadrilateral $Q(a_0,a_1,a_2,a_3)$ is bicentric and a harmonic mapping $f$ of the unit disk onto $Q(a_0,a_1,a_2,a_3)$ that solves the Beltrami equation \begin{equation}\label{beleq}\bar f_z(z) = \left(\frac{w+\frac{\imath \left(1-w^4\right) z}{\left|1-w^4\right|}}{1+\frac{\imath\overline{w} \left(1-w^4\right) z}{\left|1-w^4\right|}}\right)^2 f_z(z),\end{equation} $|z|<1$ and satisfies the initial condition $f(0)=0$, $f_z(0)>0$. The mapping $f$ gives rise to a Scherk type minimal surface $S^\diamond: \zeta=\mathbf{f}^\diamond(u,v)$ over the quadrilateral $Q(a_0,a_1,a_2,a_3)$, containing the point  $\xi=(0,0,0)$ above the origin so that its unit normal is $$\mathbf{n}^\diamond_{\xi}=-\frac{1}{1+|w|^2}(2\Im w, 2\Re w, -1+|w|^2),$$ and such that $D_{uv}\mathbf{f}^\diamond(0,0)=0$. Moreover, every other non-parametric minimal surface $S:$ $\zeta=\mathbf{f}(u,v)$ over the unit disk, containing the point $\xi$ above zero, with $\mathbf{n}_{\xi}=\mathbf{n}^\diamond_{\xi}$ and $D_{uv}\mathbf{f}(0,0)=0$ satisfies the  inequality $$|\mathcal{K}_{S}(\xi)|<|\mathcal{K}_{S^\diamond}(\xi)|.$$
\end{theorem}

%
%

To prove Theorem~\ref{mainres}, we will construct Scherk type surfaces described in Theorem~\ref{prejprej} over the bicentric quadrilaterals inscribed in the unit disk and estimate their curvature at the point above zero. Namely, we will prove the following theorem:
\begin{theorem}\label{seconda}
Assume that $Q(b_0,b_1,b_2,b_3)$ is a bicentric quadrilateral inscribed in the unit disk and assume that $0\in Q(b_0,b_1,b_2,b_3)$. Assume also that $\zeta(u,v)=(u,v,\mathbf{f}^\ast(u,v))$ is a Scherk type minimal surface above $Q(b_0,b_1,b_2,b_3)$ with $\xi$ above $0$. Then $$|\mathcal{K}(\xi)|\le \frac{\pi^2}{2}.$$
\end{theorem}
\subsection{Organization of the paper}

This section contains one more subsection, where we express the Gaussian curvature in terms of the Enneper-Weierstrass representation. The proofs of Theorem~\ref{seconda} and Theorem~\ref{prejprej} form the core of the paper. They are involved and contain a number of subtle relationships. The proof of Theorem~\ref{prejprej} is given in Section~\ref{mainma}, while the proof of Theorem~\ref{seconda} is presented in Section~\ref{sec3}.
We prepare for these proofs in Section~\ref{sec2} by describing certain Scherk-type surfaces over the unit disk. Once Theorem~\ref{prejprej} and Theorem~\ref{seconda} have been established, the proof of Theorem~\ref{mainres} follows easily.

\begin{proof}[Proof of Theorem~\ref{mainres}]
Assume that $S=\{(u,v,\mathbf{f}(u,v))\}$ is an arbitrary minimal surface above the disk $D(w_0, R)$, and assume that $\xi$ is the point above $w_0$. Assume, without loss of generality, that $R=1, w_0=0$ and $\mathbf{f}(0,0)=0.$ For $\mathbf{f}^\tau(w) = \mathbf{f}(e^{\imath\tau}w)$, we have
$$\mathbf{f}^\tau_u(0,0)=\cos \tau \,\mathbf{f}_u(0,0)+\sin \tau \,\mathbf{f}_v(0,0).$$ Furthermore
$$\mathbf{f}^\tau_{uv}(0,0)=\cos(2 \tau) \mathbf{f}_{uv}(0,0)+\cos (\tau) \sin(\tau)  \left(-\mathbf{f}_{vv}(0,0)+\mathbf{f}_{uu}(0,0)\right).$$
Since $\mathbf{f}^{\pi/2}_{uv}(0,0)=-\mathbf{f}_{uv}(0,0)$, there is $\tau$ so that $\mathbf{f}^\tau_{uv}(0,0)=0$. In other words, the new minimal surface
$S^\tau:\zeta(u,v)=(u,v,\mathbf{f}^\tau(u,v))$ satisfies the condition $\mathbf{f}^\tau_{uv}(0,0)=0$ of Theorem~\ref{prejprej}.
 Let $\mathbf{n}$ be its unit normal at $\xi=(0,0,\mathbf{f}^\tau(0,0))=(0,0,0)$. Then, from Theorem~\ref{prejprej}, there is a Scherk type surface $S^\diamond: \zeta(u,v)=(u,v,\mathbf{f}^\diamond(u,v))$, above a quadrilateral $Q(a_0,a_1,a_2,a_3)$  having the  unit normal equal to $\mathbf{n}$ at $\xi=(0,0,0)$. Moreover $$|\mathcal{K}_{S^\tau}(\mathbf{\xi})|<|\mathcal{K}_{S^\diamond}(\mathbf{\xi})|.$$
 On the other hand, by Theorem~\ref{seconda}, $|\mathcal{K}_{S^\diamond}(\mathbf{\xi})|\le \pi^2/2$ and this finishes the proof.
\end{proof}

\subsection{Minimal surfaces and Gaussian curvature}
Let $S\subset \R^3=\C\times \R$ be a minimal graph lying over the unit disc $\D\subset \C$.
Let $\varpi=(u,v,T):\D\to S$ be a conformal harmonic parameterization of $S$ with $\varpi(0)=0$.
Its projection $(u,v):\D\to \D$ is a harmonic diffeomorphism of the disc which may be assumed
to preserve orientation. Let $z$ be a complex variable in $\D$, and write
$u+\imath v = f$ in the complex notation.
We denote by $f_z=\di f/\di z$ and $f_{\bar z}=\di f/\di \bar z$ the Wirtinger derivatives of $f$.
The function $\omega$ defined by
\begin{equation}\label{eq:omega}
\overline{f_{\bar z}} = \omega f_z
\end{equation}
is called the {\em second Beltrami coefficient} of $f$, and the above equation is the
{\em second Beltrami equation} with $f$ as a solution. Observe that $\bar{f}_z=\overline{f_{\bar z}}$ and this notation will be used in the sequel.

Orientability of $f$ is equivalent to $\mathrm{Jac}(f)=|f_z|^2-|f_{\bar z}|^2>0$, hence to
$|\omega|<1$ on $\D$. Furthermore, the function $\omega$ is holomorphic whenever
$f$ is harmonic and orientation preserving. (In general, it is meromorphic when $f$ is harmonic.)
To see this, let
\begin{equation}\label{eq:fhg}
u+\imath v = f = h+\overline g
\end{equation}
be the canonical decomposition of the harmonic map $f:\D\to\D$,
where $h$ and $g$ are holomorphic functions on the disk. Then,
\begin{equation}\label{eq:omega2}
f_z=h',\quad\ f_{\bar z}=\overline g_{\bar z}= \overline{g'}, \quad\
\omega =  \overline{f_{\bar z}}/f_z = g'/h'.
\end{equation}
In particular, the second Beltrami coefficient $\omega$ equals the meromorphic function $g'/h'$
on $\D$. In our case we have $|\omega|<1$, so it is a holomorphic map $\omega:\D\to\D$.

We now consider the Enneper--Weierstrass representation of the minimal graph
$\varpi=(u,v,T):\D \to S\subset \D\times \R$ over $f$, following Duren \cite[p.\ 183]{DUREN}. We have
\begin{eqnarray*}
	u(z) &=& \Re f(z) = \Re \int_0^z \phi_1(\zeta)d\zeta \\
	v(z) &=& \Im f(z) = \Re \int_0^z \phi_2(\zeta)d\zeta \\
	T(z) &=& \Re \int_0^z \phi_3(\zeta)d\zeta
\end{eqnarray*}
where
\begin{eqnarray*}
	\phi_1 &=& 2(u)_z = 2(\Re f)_z = (h+\bar g + \bar h + g)_z = h'+g', \\
	\phi_2 &=& 2(v)_z = 2(\Im f)_z = \imath(\bar h+g - h -\bar g)_z = \imath(g'-h'), \\
	\phi_3 &=& 2(T)_z = \sqrt{-\phi_1^2-\phi_2^2} = \pm 2\imath \sqrt{h'g'}.
\end{eqnarray*}
The last equation follows from the identity $\phi_1^2+\phi_2^2+\phi_3^2=0$ which
is satisfied by the Enneper--Weierstrass datum $\phi=(\phi_1,\phi_2,\phi_3)=2\di \varpi.$ Let us denote $\mathbf{p}=f_z$. We have
\begin{equation}\label{eq:p}
\mathbf{p} = f_z = (\Re f)_z + \imath (\Im f)_z = \frac12(h'+g' + h'-g') = h'.
\end{equation}
By using $\omega =  \overline{f_{\bar z}}/f_z = g'/h'$ (see \eqref{eq:omega2}), it follows that
\[
\phi_1 = h'+g'=\mathbf{p}(1+\omega),\quad \phi_2 = -\imath(h'-g')=-\imath \mathbf{p}(1-\omega),\quad
\phi_3 = \pm 2\imath \mathbf{p} \sqrt{\omega}.
\]
From the formula for $\phi_3$ we infer that $\omega$ has a well-defined holomorphic square root:
\begin{equation}\label{eq:q}
\omega = \mathbf{q}^2,\qquad \mathbf{q}:\D\to \D\ \ \text{holomorphic}.
\end{equation}
In terms of the Enneper--Weierstrass parameters $(\mathbf{p},\mathbf{q})$ given by \eqref{eq:p} and \eqref{eq:q}, we obtain
\begin{equation}\label{eq:EW}
\phi_1 = \mathbf{p}(1+\mathbf{q}^2),\quad \phi_2 = -\imath \mathbf{p}(1-\mathbf{q}^2),\quad
\phi_3 = -2\imath \mathbf{p} \mathbf{q}.
\end{equation}
(The choice of sign in $\phi_3$ is a matter of convenience; since we have two choices of sign for
$\mathbf{q}$ in \eqref{eq:q}, this does not cause any loss of generality. Hence,
\begin{equation}\label{param}
\varpi(z) = \left(\Re f(z), \Im f(z), \Im \int_0^z 2 \mathbf{p}(\zeta) \mathbf{q}(\zeta) d\zeta \right),\quad z\in\D.
\end{equation}

The curvature $\mathcal{K}$ of the minimal graph $S$ is expressed in terms of $(h,g,\omega)$ \eqref{eq:omega2}, and in terms
of the Enneper--Weierstrass parameters $(\mathbf{p},\mathbf{q})$, by
\begin{equation}\label{eq:curvatureformula}
\mathcal{K} = - \frac{|\omega'|^2}{|h'g'|(1 + |\omega|)^4} = - \frac{4|\mathbf{q}'|^2}{|\mathbf{p}|^2(1 + |\mathbf{q}|^2)^4},
\end{equation}
where $\mathbf{p}=f_z$ and $\omega=\mathbf{q}^2=\overline{f_{\bar z}}/f_z$. (See Duren \cite[p.\ 184]{DUREN}.) For a slightly different formula concerning the expression of Gaussian curvature we refer to the monographs by Alarcon, Forstneri\v c and Lopez \cite[Sec.~2.6]{ALARCONFORSTNERICLOPEZ} and also the paper of P\'erez and Ros in the monograph \cite{ROS}.
\section{Proof of Theorem~\ref{prejprej}}\label{mainma}
\begin{proof}[Proof of Theorem~\ref{prejprej}]
In order to prove Theorem~\ref{prejprej}, we will derive a useful formula for $\mathbf{f}_{uv}$, of a non-parametric minimal surface $T=\mathbf{f}(u,v)$. Namely we will express $\mathbf{f}_{uv}$ as a function of Enneper-Weisstrass parameters.
Assume that $\mathbf{q}(z) = a(z) + \imath b(z)=\sqrt{\omega(z)}$ and $\mathbf{p}$ are Enneper-Weisstrass parameters of a minimal disk $S=\{(u(z), v(z), T(z)), z\in \D\}=\{(u,v,\mathbf{f}(u,v)): (u,v)\in \D\}$ over the unit disk. Here$, z=x+\imath y, f=u+\imath v$ and $\bar f_z=\omega(z) f_z$.

From Enneper--Weierstrass parameterization, we have that $\frac{\partial \varpi}{\partial x}=(\frac{\partial u}{\partial x}, \frac{\partial v}{\partial x}, \frac{\partial T}{\partial x})=\Re(\phi_1, \phi_2, \phi_3)$ and $\frac{\partial \varpi}{\partial y}=(\frac{\partial u}{\partial y}, \frac{\partial v}{\partial y}, \frac{\partial T}{\partial y})=-\Im( \phi_1, \phi_2, \phi_3)$ and the cross product is $\Im(\phi_2\overline{\phi_3},\phi_3\overline{\phi_1}, \phi_1\overline{\phi_2})=-|\mathbf{p}|^2(1+|\mathbf{q}|^2)(2\Im \mathbf{q},2\Re \mathbf{q}, -1+|\mathbf{q}|^2).$ Then, in view of \cite[p.~169]{DUREN}, the unit normal at $\mathbf{\xi}\in S$ is given by $$\mathbf{n}_{\mathbf{\xi}}=-\frac{1}{1+|\mathbf{q}(z)|^2}(2\Im \mathbf{q}(z), 2\Re \mathbf{q}(z), -1+|\mathbf{q}(z)|^2).$$ Note that $\mathbf{q}$ is the Gauss map of the minimal surface, after the stereographic projection. The unit normal is also given by the formula
$$\mathbf{n}_{\mathbf{\xi}}=\frac{1}{\sqrt{1+\mathbf{f}_u^2+\mathbf{f}_v^2}}\left(-\mathbf{f}_u,-\mathbf{f}_v,1\right).$$  Then we have the relations
 \begin{equation}\label{firsti}\mathbf{f}_v (u(x,y),v(x,y))=\frac{2 a(x,y)}{-1+a(x,y)^2+b(x,y)^2},
  \end{equation}
  \begin{equation}\label{secondi}\mathbf{f}_u(u(x,y),v(x,y))=\frac{2 b(x,y)}{-1+a(x,y)^2+b(x,y)^2}.\end{equation}
By differentiating \eqref{firsti} and \eqref{secondi}  w.r.t. $x$  we obtain the equations
 \begin{equation}\label{firsti1}v_x \mathbf{f}_{uv}(u,v)+u_x \mathbf{f}_{uu}(u,v)=-\frac{4 a b a_x+2 \left(1-a^2+b^2\right) b_x}{\left(-1+a^2+b^2\right)^2},\end{equation}
  \begin{equation}\label{secondi1}v_x\mathbf{f}_{vv}(u,v) +u_x \mathbf{f}_{uv}(u,v)=-\frac{4 a b b_x+2 \left(1-a^2+b^2\right) a_x}{\left(-1+a^2+b^2\right)^2}.\end{equation}
Now recall the minimal surface equation
 \begin{equation}\label{mse}\left(1+\mathbf{f}^2_u(u,v)^2\right)\mathbf{f}_{vv}(u,v)+\left(1+\mathbf{f}^2_v(u,v)^2\right) \mathbf{f}_{uu}(u,v)=2 \mathbf{f}_v(u,v) \mathbf{f}_u(u,v)  \mathbf{f}_{uv}(u,v). \end{equation}
 From \eqref{firsti}, \eqref{secondi}, \eqref{firsti1},\eqref{secondi1} and \eqref{mse} we get
$$\mathbf{f}_{uv}=\frac{M}{N}$$ where
\[\begin{split}M&=-2 (a^4+2 a^2 (-1+b^2)+(1+b^2)^2) ((1+a^2-b^2) a_x+2 a b b_x) u_x\\&-2((1+a^2)^2+2 (-1+a^2) b^2+b^4) (2 a b a_x+(1-a^2+b^2) b_x) v_x\end{split}\] and
\[\begin{split}N&=(1-a^2-b^2)^2 \\&\times ((a^4-2 a^2 (1-b^2)+(1+b^2)^2) u_x^2+8 a b u_x v_x+((1+a^2)^2-2 (1-a^2) b^2+b^4) v_x^2).\end{split}\]
Let $\mathbf{q}(z)=a+\imath b=r e^{it}$, $\mathbf{q}'(z) = a_x+\imath b_x=R e^{\imath s}$ and $\mathbf{p}=Pe^{\imath m}$.  Because $u_x=\Re (\mathbf{p}(1+\mathbf{q}^2))$, and $v_x=-\Re (\imath \mathbf{p}(1-\mathbf{q}^2))$,
after straightforward calculation we get
$$\mathbf{f}_{uv}=-\frac{2 R \left(\cos[m-s]-r^4 \cos[m-s+4 t]\right)}{P \left(1-r^2\right)^3 \left(1+r^2\right)},$$
which can be written as
\begin{equation}\label{fexpli}\mathbf{f}_{uv}=-\frac{2\Re \left[\mathbf{p}(1-\mathbf{q}^4)\overline{\mathbf{q}'}\right]}{|\mathbf{p}|^2(1-|\mathbf{q}|^2)^3(1+|\mathbf{q}|^2)}.\end{equation}

Now we continue to prove Theorem~\ref{prejprej}.
The solution of \eqref{beleq} with such initial conditions exists and is unique \cite[Theorem~A\&~Theorem~1]{BSHOUTYLYZZAIKWEITSMAN} and maps the unit disk onto a quadrilateral $Q(a_0,a_1,a_2,a_3)$ whose vertices $a_0,a_1,a_2,a_3$, $a_4=a_0$ belongs to the unit circle. Moreover by  \cite[Theorem~B]{BSHOUTYLYZZAIKWEITSMAN}, there are four points $b_k=e^{\imath\beta_k}, \ k=0,1,2,3$, $b_4=b_0, b_5=b_1$, $$F(e^{\imath t})=\sum_{k=1}^4 a_k I_{(\beta_k, \beta_{k+1})}(t).$$ Here $F$ is the boundary function of $f$ and $I_{(\beta_k, \beta_{k+1})}(t)$ function defined on $\mathbb{T}$ as:
$$I_{(\beta_k, \beta_{k+1})}(t)=\begin{cases}
1, \quad\beta_k<t<\beta_{k+1}\\
0,\quad t<\beta_k \quad\text{or}\quad t>\beta_{k+1}.
\end{cases}$$ Therefore (see \cite[p.~63]{DUREN}) we can conclude that $$f_z(z) = \sum_{k=1}^4 \frac{d_k}{z-b_k}$$  and that  $$\bar{f}_z(z) = -\sum_{k=1}^4 \frac{\overline{d_k}}{z-b_k},$$ where
$$d_k = \frac{a_k -a_{k+1}}{2\pi \imath}.$$
Hence  the third coordinate of conformal parameterization is
 $$T(z) =\pm 2\Re \imath \int_0^z \sqrt{f_z\bar f_z}dz.$$ Thus when $z$ is close to $b_k$, then \begin{equation}\label{tinf}T(z) =\pm|d_k|^2\log|1-z/b_k|+O(z-b_k).\end{equation}
This implies that  $z\to b_k$, $T(z)\to \pm \infty$. This implies that $\mathbf{f}(z)\to \pm\infty$ if $z\to a\in(a_k, a_{k+1})$.
Since $$\mathbf{q}(z) =\frac{w+\frac{\imath \left(1-w^4\right) z}{\left|1-w^4\right|}}{1+\frac{\imath\overline{w} \left(1-w^4\right) z}{\left|1-w^4\right|}},$$ we get $$\mathbf{q}(0) = w \ \ \text{and}\ \ \mathbf{q}'(0)=\frac{\imath \left(1-w^4\right) \left(1-|w|^2\right) }{\left|1-w^4\right|}.$$

Further
$$\mathbf{p}(0)(1-\mathbf{q}(0)^4)\overline{\mathbf{q}'(0)}=-\imath f_z(0)|1-w^4|(1-|w|^2).$$
So in view of the formula \eqref{fexpli} we conclude $\mathbf{f}^\diamond_{uv}=0.$

Now we assert that \begin{equation}\label{needed}|\mathcal{K}_{S}(\mathbf{\xi})|< |\mathcal{K}_{S^{\diamond}}(\mathbf{\xi})|.\end{equation}
  Assume the converse $|\mathcal{K}_{S}(\mathbf{\xi})|\ge  |\mathcal{K}_{S^{\diamond}}(\mathbf{\xi})|$ and argue by a contradiction.  Then as in \cite{FINNOSSERMAN}, by using the dilatation $L(\zeta) = \lambda \zeta$ for some $\lambda\ge 1$ we get the surface
$$S_1=L(S)=\{(u,v,\lambda \mathbf{f}\left(\frac{u}{\lambda}, \frac{v}{\lambda}\right): |u+\imath v|<{\lambda}\},$$ whose Gaussian curvature
$$\mathcal{K}_1(\mathbf{\xi})=\frac{\frac{1}{\lambda^2} \left(\mathbf{f}_{uu}(0,0)\mathbf{f}_{vv}(0,0)-\mathbf{f}_{uv}(0,0)^2\right)}{(1+\mathbf{f}_u(0,0)^2+\mathbf{f}_v(0,0)^2)^2}.$$ Observe that such transformation does not change the unit normal at $\mathbf{\xi}$.

Then there is $\lambda_\ast\ge 1$ so that $\mathcal{K}_1(\mathbf{\xi})=\mathcal{K}_{S^\diamond}(\mathbf{\xi})$. Let $$\mathbf{f}^\ast (u,v) =\lambda_\ast \mathbf{f}\left(\frac{u}{\lambda_\ast}, \frac{v}{\lambda_\ast}\right).$$
From
$\mathbf{n}_\diamond=\mathbf{n}_\ast,$ we get
\begin{equation}\label{nowafter}\mathbf{f}^\diamond_{u}(0,0)=\mathbf{f}^\ast_{u}(0,0),  \ \mathbf{f}^\diamond_{v}(0,0)=\mathbf{f}^\ast_{v}(0,0).\  \end{equation}

  Further we have $$(1+(\mathbf{f}^\ast_{u}(0,0))^2)\mathbf{f}^\ast_{vv} (0,0)-2 \mathbf{f}^\ast_{u}(0,0)\mathbf{f}^\ast_{v}(0,0)\mathbf{f}^\ast_{uv} (0,0)+(1+(\mathbf{f}^\ast_{v}(0,0))^2)\mathbf{f}^\ast_{uu} (0,0)=0,$$
$$(1+(\mathbf{f}^\diamond_{u}(0,0))^2)\mathbf{f}^\diamond_{vv} (0,0)-2 \mathbf{f}^\diamond_{u}(0,0)\mathbf{f}^\diamond_{v}(0,0)\mathbf{f}^\diamond_{uv} (0,0)+(1+(\mathbf{f}^\diamond_{v}(0,0))^2)\mathbf{f}^\diamond_{uu} (0,0)=0,$$  $$ \mathbf{f}^\diamond_{uv}(0,0)=\mathbf{f}^\ast_{uv}(0,0)$$ and the equation
$$\frac{ \left(\mathbf{f}^\ast _{uu}(0,0)\mathbf{f}^\ast_{vv}(0,0)-\mathbf{f}^\ast_{uv}(0,0)^2\right)}{(1+\mathbf{f}^\ast_u(0,0)^2+\mathbf{f}^\ast_v(0,0)^2)^2}=
\frac{ \left(\mathbf{f}^\diamond _{uu}(0,0)\mathbf{f}^\diamond_{vv}(0,0)-\mathbf{f}^\diamond_{uv}(0,0)^2\right)}{(1+\mathbf{f}^\diamond_u(0,0)^2+\mathbf{f}^\diamond_v(0,0)^2)^2}.$$
We can also w.l.g. assume that $\mathbf{f}^\ast_{uu} $ and $\mathbf{f}^\diamond_{uu} $ as well as $\mathbf{f}^\ast_{vv} $ and $\mathbf{f}^\diamond_{vv} $ have the same sign. If not, then we choose $\lambda_\ast\le -1$ and repeat the previous procedure with $$S_1=L(S)=\{(u,v,\lambda \mathbf{f}\left(\frac{u}{\lambda}, \frac{v}{\lambda}\right): |u+\imath v|<{|\lambda|}\}.$$
Thus the function $F(u,v) = \mathbf{f}^\ast(u,v)- \mathbf{f}^\diamond(u,v)$ has all derivatives up to the order $2$ equal to zero in the point $w=0$.

To continue the proof we use the following lemma
\begin{lemma}\label{leci}
Assume that the quadrilateral $Q=Q(a,b,c,d)$ is inscribed in the unit disk, and assume that $\zeta=\mathbf{f}(u,v)$ is a Scherk type minimal surface $S$ above $Q$. i.e. assume that $\mathbf{f}(u,v)\to +\infty$ when $\zeta=u+\imath v \to w\in  (a,b)\cup (c,d)$ and $\mathbf{f}(u,v)\to -\infty$ when $\zeta=u+\imath v \to w\in  (b,c)\cup (a,d)$. Then there is not any other bounded minimal graph $\zeta=\mathbf{f}_1(u,v)$ over a simply connected domain $\Omega$ that contains  $Q$ which has the same Gaussian curvature, the same unit normal, and the same mixed derivative  at the same point $\mathbf{\xi}\in Q$ as the given surface $S$.
\end{lemma}
\begin{proof}[Proof of Lemma~\ref{leci}]
We observe that \cite[Proof of Proposition~1]{FINNOSSERMAN} works for every Scherk type minimal surface, so if we would have a  bounded minimal surface having the all derivatives up to the order 2 equal to zero, then such non-parametric parameterizations $\mathbf{f}$ and $\mathbf{f}_1$, in view of \cite[Lemma~1]{FINNOSSERMAN} will satisfy the relation $F(z)=\mathbf{f}(z)-\mathbf{f}_1(z)=O(\zeta^N(z))$, $N\ge 3$, where $\zeta$ is a certain homeomorphism between two open sets containing $0$. Then by following the proof of \cite[Proof of Proposition~1]{FINNOSSERMAN} (second part) we get that this is not possible, because Scherk type surface has four "sides" but the number $2N$ is bigger or equal to $6$ which is not possible.
\end{proof}

This leads to the contradiction so \eqref{needed} is true. This finishes the proof of Theorem~\ref{prejprej}.
\end{proof}

\section{Two parameter family of Scherk type minimal surfaces}\label{sec2}

 In this section, we provide a complete description of Scherk-type minimal surfaces defined over bicentric quadrilaterals inscribed in the unit disk.
According to results of Jenkins and Serrin~\cite{JENKINSSERRIN}, as well as Sheil-Small~\cite{SHEILSMALL}, the image of the corresponding harmonic function
$f$ is a bicentric quadrilateral. Following this construction, we explicitly compute certain quantities needed for the evaluation of the Gaussian curvature.

Harmonic mappings with prescribed dilatation (especially when the Beltrami coefficient is a Blaschke product, as in our case) have been studied extensively. Several results concerning the existence and uniqueness of such mappings, as well as more general classes, can be found in~\cite{BSHOUTYLYZZAIKWEITSMAN},\cite{HENGARTNERSCHOBER},\cite{JENKINSSERRIN}, \cite{JENKINSSERRIN1}, and~\cite{SHEILSMALL}.
We begin with this proposition.
\begin{proposition}\label{propduke}\cite{JENKINSSERRIN}
Let \( Q \subset \mathbb{R}^2 \) be a convex quadrilateral with boundary edges labeled cyclically as \( A_1, B_1, A_2, B_2 \). Suppose the boundary data for the minimal surface Dirichlet problem is prescribed by
\begin{equation}\label{conduke}
T = +\infty \quad \text{on } A_1 \cup A_2, \qquad T = -\infty \quad \text{on } B_1 \cup B_2.
\end{equation}
Then the Dirichlet problem for the minimal surface equation over \( Q \) admits a solution if and only if the following identity is satisfied:
\[
|A_1| + |A_2| = |B_1| + |B_2|,
\]
where \( |\cdot| \) denotes the Euclidean length of each edge.

Moreover, when a solution exists, it is unique up to an additive constant; that is, any two solutions differ by a vertical translation.
\end{proposition}

\medskip

Now let \( Q \subset \mathbb{R}^2 \) be a convex quadrilateral with sides labeled cyclically as \( A_1, B_1, A_2, B_2 \), and suppose that \( Q \) is inscribed in the unit disk \( \mathbb{D} \subset \mathbb{C} \). We construct in Subsection~\ref{subse28} a harmonic diffeomorphism \( f : \mathbb{D} \to Q \) that solves the second Beltrami equation:
\[
\overline{f_{\bar{z}}} = \nu(z)\, f_z, \qquad \nu(z) =  \left(e^{i t} \frac{z - a}{1 - z \bar{a}} \right)^2,
\]
for some \( a \in \mathbb{D} \) and angle \( t \in \mathbb{R} \).

This mapping \( f \) provides, through \eqref{param}, a conformal parameterization of the minimal graph
\[
S = \left\{ (u, v, T_1(u, v)) : (u, v) \in Q \right\},
\]
which satisfies the boundary condition~\eqref{conduke} (this follows from the asymptotic behavior of the parametrization; see \eqref{tinf}).

Since both \( T \) and \( T_1 \) solve the same Dirichlet problem over \( Q \) with identical infinite boundary conditions, the uniqueness result from Proposition~\ref{propduke} implies that
\[
T_1 = T + c
\]
for some constant \( c \in \mathbb{R} \). Thus, every such minimal graph over \( Q \) arises from this construction, up to a vertical translation.

\subsection{Construction of the two-parameter family}\label{subse28}
Note that if $f$ satisfies \eqref{beleq} then
$$f^{\star}(z):=f\bigg(-\frac{e^{-\imath\sigma}(w-z)}{1-\overline{w}z}\bigg),$$
where $e^{\imath\sigma}=\frac{\imath(1-w^4)}{|1-w^4|},$ satisfies
$$\overline{f^{\star}}_{ \bar{z}}=z^2 f^{\star}_{ z},$$
with $f^{\star}(w)=0,\quad  \imath(1-w^4)f^{\star}_{z}(w)>0.$ We will, therefore, consider the following Beltrami equation:
\begin{equation}
\label{eq:belz^2}
\overline{f}_{ \bar{z}}=z^2 f_{ z}.
\end{equation}
 Its solution is well understood and already described in papers \cite{JENKINSSERRIN},\cite{BSHOUTYLYZZAIKWEITSMAN} and~\cite{SHEILSMALL}. We will not impose these additional conditions at this stage, as omitting them allows for a slightly broader family of Scherk-type surfaces. These conditions can be easily incorporated at the end of the construction to recover exactly the family described in \ref{prejprej}.
 If $f_1$ is harmonic mapping and it satisfies the equation \eqref{eq:belz^2}, by \cite{BSHOUTYLYZZAIKWEITSMAN} and \cite{SHEILSMALL}, it is a mapping of the unit disk onto a quadrilateral inscribed in the unit disk, which is a Poisson extension of a step function, determined by a set of four points on the unit circle, that defines a quadrilateral in the domain. In \cite{SHEILSMALL} it is proved that the sums of lengths of two non-adjacent sides of quadrilateral in co-domain are equal, i.e. such quadrilateral is bicentric. Moreover, in Example~4 from the same paper the relation between the vertices of quadrilaterals in domain and co-domain is explicitly given. More precisely, if $b_0=1, b_1=e^{\imath \beta_1}, b_2=e^{\imath \beta_2}, b_3=e^{\imath \beta_3}$ and $f_1$ is defined as the harmonic extension of the function

  $$F_1(\psi)=\left\{
 \begin{array}{ll}
    a_1=e^{\imath \alpha_1}, & \hbox{if $\psi\in[0,\beta_1)$;} \\
 a_2=e^{\imath \alpha_2}, & \hbox{if $\psi\in[\beta_1,\beta_2)$;} \\
  a_3=e^{\imath \alpha_3}, & \hbox{if $\psi\in[\beta_2,\beta_3)$;} \\
  a_4=e^{\imath \alpha_4}, & \hbox{if $\psi\in[\beta_3,2\pi )$,}
  \end{array}
  \right.
  $$
 by the mentioned example from \cite{SHEILSMALL}, we have:
$$ b_k=\epsilon_k \frac{|a_{k+1}-a_k|}{a_{k+1}-a_k}$$
for $k=1,2,3,4,$ $\epsilon_1=\epsilon_3=1, \epsilon_2=\epsilon_4=-1,$ $\alpha_5=\alpha_1+2\pi.$ From
$$a_{k+1}-a_k=e^{\imath \alpha_{k+1}}-e^{\imath \alpha_k}=2\imath\sin\frac{\alpha_{k+1}-\alpha_k}{2} e^{\imath \frac{\alpha_{k+1}+\alpha_k}{2}},$$
we conclude $b_0b_2+b_1b_3=0$ and
\begin{equation}
  \label{eq:temena}
  b_k=e^{\imath \big(\frac{3\pi}{2}-\frac{\alpha_{k+1}+\alpha_k}{2}\big)}.
  \end{equation}
Hence, if we take $\beta_1=p,$ $\beta_2=q,$ then $\beta_3=\pi+q-p$,
\begin{equation}\label{eq:condition1}0<p< q<q-p+\pi<2\pi.\end{equation} Also, from the equation \eqref{eq:temena} we can express $\alpha_k'$s in terms of $p, q$ and  $\beta $ as $\alpha_1=\frac{q+\pi}{2}+\beta-p,$ $\alpha_2=\frac{5\pi-q}{2}-p-\beta,$ $\alpha_3=\frac{5\pi-3q}{2}+\beta+p$ and $\alpha_4=\frac{5\pi-q}{2}+p-\beta$.

  After the rotation $\mathcal{R}$ of the quadrilateral in co-domain for the angle $\beta+\frac{q-\pi}{2},$ we get the quadrilateral with vertices $e^{\imath x}, e^{\imath y}, e^{\imath s}, e^{\imath p},$ where $x=q-p+2\beta, y=2\pi-p, s=2\pi-q+p+2\beta,$ and $\beta$ is a certain parameter. We will show that $\beta$ depends on $p$ and $q.$ In order to do this, we start from the identity:
$$|e^{\imath p}- e^{\imath x}|+|e^{\imath y}-e^{\imath s}|=|e^{\imath x}- e^{\imath y}|+|e^{\imath s}-e^{\imath p}|,$$
or, equivalently:
$$\bigg|\sin\frac{p-x}{2}\bigg|+\bigg|\sin\frac{y-s}{2}\bigg|=\bigg|\sin\frac{x-y}{2}\bigg|+\bigg|\sin\frac{s-p}{2}\bigg|,$$
since the quadrilateral with vertices $e^{\imath x}, e^{\imath y}, e^{\imath s}, e^{\imath p},$ is also bicentric.
Inserting  $x=q-p+2\beta, y=2\pi-p, s=2\pi-q+p+2\beta,$ we get
$$|\sin(\beta+\frac{q}{2}-p)|+|\sin(\beta+p-\frac{q}{2})|=|\sin(\frac{q}{2}+\beta-\pi)|+|\sin(\pi-\frac{q}{2}+\beta)|.$$
To transform this equality in more conventional form, we recall the ordering between vertices. In fact, the following chain of inequalities holds:
\begin{equation}
\label{eq:poredak}
p< q-p+2\beta<2\pi-p<2\pi-q+p+2\beta<2\pi+p,
\end{equation}
which implies the following conditions: $\beta+p >\frac{q}{2}, \beta+\frac{q}{2}> p,$ $\frac{q}{2}+\beta<\pi$ and $\frac{q}{2}>\beta.$
Further, we infer $\beta> 0$ and $\beta+\frac{q}{2}-p \in(0,\pi),$ $\beta+\frac{q}{2}\in (0,\pi)$ and $0<\beta+p-\frac{q}{2}=\beta+\frac{q}{2}+p-q<\beta+\frac{q}{2}<\pi.$

Hence $$|\sin(\beta+\frac{q}{2}-p)|=\sin(\beta+\frac{q}{2}-p),\ \  |\sin(\frac{q}{2}+\beta-\pi)|=\sin(\frac{q}{2}+\beta),$$
$$|\sin(\beta+p-\frac{q}{2})|=\sin(\beta+p-\frac{q}{2})\text{ and } \ \ |\sin(\pi-\frac{q}{2}+\beta)|=\sin(\pi-\frac{q}{2}+\beta),$$ therefore, we have $$\sin(\beta+\frac{q}{2}-p)+\sin(\beta+p-\frac{q}{2})=\sin(\beta+\frac{q}{2})+\sin(\frac{q}{2}-\beta)$$ i.e.  $$\sin\beta\cos(\frac{q}{2}-p)=\cos\beta\sin\frac{q}{2},$$ which finally implies
\begin{equation}
\label{eq:tanbeta}
\tan\beta=\frac{\sin\frac{q}{2}}{\cos(\frac{q}{2}-p)}.
\end{equation}
Note that $\tilde{f}:=\Rcal(f_1(z))$ has the dilatation \begin{equation}\label{dila}\tilde\omega(z)=\frac{\tilde g'}{\tilde h'}=-e^{-\imath(2\beta+q)}z^2.\end{equation}
We are now searching for M\"{o}bius transformation $M$ which maps $(z_1,z_2,z_3,z_4)=(1, e^{\imath \alpha}, -1, -e^{\imath \alpha})$ to
$(w_1,w_2,w_3,w_4)=(1, e^{\imath p}, e^{\imath q}, -e^{\imath(q-p)}).$ By  using the formula
$$\frac{(z_1-z_3)(z_2-z)}{(z_2-z_3)(z_1-z)}=\frac{(w_1-w_3)(w_2-w)}{(w_2-w_3)(w_1-w)},$$
where $w_j=M(z_j), j=1,2,3,$ and  $ w=M(z),$ we get:
$$\frac{2(e^{\imath\alpha}-z)}{(e^{\imath \alpha}+1)(1-z)}=\frac{(1-e^{\imath q})(e^{ip}-M(z))}{(e^{\imath p}-e^{\imath q})(1-M(z))},$$
i.e.:
\begin{equation}
\label{eq:mebijus}
 M(z)= 1+\frac{z-1}{\frac{2(z-e^{\imath\alpha})}{(e^{\imath\alpha}+1)(e^{\imath q}-1)}+\frac{(e^{\imath\alpha}-1)(z+1)}{(e^{\imath\alpha}+1)(e^{\imath p}-1)}}.
\end{equation}
It remains to choose a specific $\alpha$  so that $M(-e^{\imath \alpha})=-e^{\imath(q-p)}$. Using equation \eqref{eq:mebijus}, with $z=-e^{\imath \alpha}$ and $M(z)=-e^{\imath(q-p)}$, we get
$$ \frac{4e^{\imath \alpha}}{(e^{\imath\alpha}+1)^2}=\frac{(1-e^{\imath q})(e^{\imath p}+e^{\imath(q-p)})}{(e^{\imath p}-e^{\imath q})(1+e^{\imath(q-p)})},$$
i.e. after inverting both sides and reducing to sine and cosine functions
$$   \cos^2\frac{\alpha}{2}=\frac{\sin(q-p)}{2\sin(\frac{q}{2})\cdot\cos(\frac{q}{2}-p)}.$$
The last identity leads to the following identity
\begin{equation}
\label{eq:kosalpha}
\begin{split}
\cos\alpha&=\frac{\sin(q-p)-\sin(\frac{q}{2})\cos(\frac{q}{2}-p)}{\sin(\frac{q}{2})\cos(\frac{q}{2}-p)}
\\&=\frac{\sin(q-p)-\sin p}{2\sin(\frac{q}{2})\cos(\frac{q}{2}-p)}\\
&=\frac{\tan(\frac{q}{2}-p)}{\tan(\frac{q}{2})}
\\&=\frac{\sin(q-p)-\sin p}{\sin(q-p)+\sin p}.
\end{split}
\end{equation}

Assume now that \begin{equation}\label{QQ} \mathcal{S}=\{(p,q): (p,q)\text{ satisfies } \eqref{eq:condition1},\eqref{eq:poredak}\text{ and }\eqref{eq:tanbeta} \text{ and }\ 0\in Q(e^{\imath p}, e^{\imath x}, e^{\imath y}, e^{\imath s})\}.
 \end{equation}
The condition \( 0 \in Q(e^{\imath p}, e^{\imath x}, e^{\imath y}, e^{\imath s}) \) is equivalent to requiring that the argument of the difference between any pair of adjacent vertices is less than \( \pi \). That is, for
\[
x = q - p + 2\beta, \quad y = 2\pi - p, \quad s = 2\pi - q + p + 2\beta,
\]
the following inequalities must hold:
\[
x - p < \pi, \quad y - x < \pi, \quad s - y < \pi, \quad p - (s - 2\pi) < \pi.
\]
The set $\mathcal{S}$ is non-empty open set (thus having positive Lebesgue measure). See Figure~\ref{25}.
\begin{figure}
\centering
\includegraphics{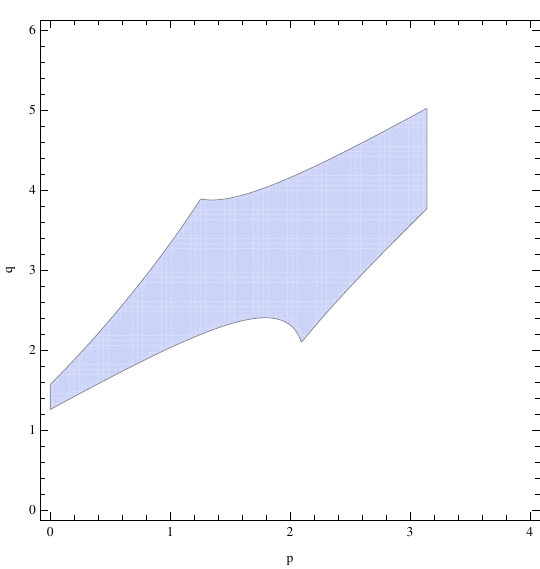}
\caption{The domain $\mathcal{S}$ is inside of the square $[0,2\pi]\times [0,2\pi]$.}
\label{25}
\end{figure}

From now on, we will assume that $(p,q)\in \mathcal{S}$ and consider the mapping $f(z):=\tilde{f}( M(z)).$ It maps the unit disk onto a bicentric quadrilateral whose vertices are $e^{\imath p},$ $e^{\imath x}$, $ e^{\imath y},$ and $ e^{\imath s}.$ Note that the points from the unit circle that are the limit points for the mapping $f$ are $1,e^{\imath \alpha},-1,-e^{\imath \alpha}$ and make a rectangle. More precisely, $$f(re^{it}) =\frac{1}{2\pi}\int_0^{2\pi}\frac{1-r^2}{1+r^2-2 r \cos(t-\psi)}F(e^{\imath\psi})d\psi$$ where $F$ is the step function defined by $$F(e^{\imath\psi})=\left\{
                     \begin{array}{ll}
                       e^{\imath x}, & \hbox{if $\psi\in[0,\alpha)$;} \\
                       e^{\imath y}, & \hbox{if $\psi\in[\alpha,\pi)$;} \\
                       e^{\imath s}, & \hbox{if $\psi\in[\pi,\pi+\alpha)$;} \\
                       e^{\imath p}, & \hbox{if $\psi\in[\pi+\alpha,2\pi)$.}
                     \end{array}
                   \right.
 $$
 From \cite{SHEILSMALL}, we again infer the existence of $a \in \mathbb{D},b\in \mathbb{C}, \theta \in [0,2\pi] $ such that
 \begin{equation}
 \label{eq:pbold}
 \mathbf{p}=\frac{b(1-z\overline{a})^2}{(z^2-1)(z^2-e^{2\imath\alpha})}
 \end{equation}
 and
 \begin{equation}
 \label{eq:qbold}
 \mathbf{q}=e^{\imath\theta}\frac{z-a}{1-z\overline{a}},
 \end{equation}
 where $\mathbf{p}=h'$ and $\mathbf{p}\mathbf{q}^2=g'.$ Note that $a$ is the zero of the M\"obius transformation $M$.

 In the next lemma we will give the explicit formula for mapping $f.$

 \begin{lemma}
 	The mapping $f$ given by the Poisson extension of the step function $F$ can be represented as:
 	\begin{equation}\label{fff}f(z)=u(z)+\imath v(z)+f(0),\end{equation}
 	with
 	$$f(0)=\frac{\alpha e^{2\imath\beta}\cos(p-q)+(\pi-\alpha)\cos p}{\pi},$$
 	\begin{equation}
 	\label{eq:u}
 	\begin{split}
 	u(re^{\imath t}) &=\frac{\cos (q-p+2\beta)-\cos p}{\pi}\tan^{-1}\frac{r\sin (\alpha-t)}{1-r\cos(\alpha-t)}
 	\\&+\frac{\cos p-\cos (p-q+2\beta)}{\pi}\tan^{-1}\frac{r\sin t}{1+r\cos t}
 	\\&+\frac{\cos p -\cos (p-q+2\beta)}{\pi}\tan^{-1}\frac{r\sin (\alpha-t)}{1+r\cos(\alpha-t)}
 	\\&+\frac{\cos (q-p+2\beta)-\cos p}{\pi}\tan^{-1}\frac{r\sin t}{1-r\cos t}\end{split}
 	\end{equation}
 	and
 	\begin{equation}
 	\label{eq:v}
 	\begin{split}
 	v(re^{\imath t}) &=\frac{\sin (q-p+2\beta)+\sin p}{\pi}\tan^{-1}\frac{r\sin (\alpha-t)}{1-r\cos(\alpha-t)}
 	\\&-\frac{\sin p+\sin (p-q+2\beta)}{\pi}\tan^{-1}\frac{r\sin t}{1+r\cos t}
 	\\&+\frac{\sin p-\sin (p-q+2\beta)}{\pi}\tan^{-1}\frac{r\sin (\alpha-t)}{1+r\cos(\alpha-t)}
 	\\&+\frac{\sin (q-p+2\beta-\sin p)}{\pi}\tan^{-1}\frac{r\sin t}{1-r\cos t}.\end{split}
 	\end{equation}
 \end{lemma}

 \begin{proof}Following \cite[Section 4.3,~ p.~63]{DUREN} and the definition of $f$ as the harmonic extension of the step function $F$, we find
 	
 	$$h'(z)=\frac{1}{2\pi \imath}\bigg(
 	\frac{e^{\imath x}-e^{\imath y}}{z-e^{\imath \alpha}}+\frac{e^{\imath y}-e^{\imath s}}{z+1}+\frac{e^{\imath s}-e^{\imath p}}{z+e^{\imath \alpha}}+\frac{e^{\imath p}-e^{\imath x}}{z-1}\bigg)$$
 	
 	and
 	
 	$$g'(z)=\frac{1}{2\pi \imath}\bigg(
 	\frac{e^{-\imath x}-e^{-\imath y}}{z-e^{\imath \alpha}}+\frac{e^{-\imath y}-e^{-\imath s}}{z+1}+\frac{e^{-\imath s}-e^{-\imath p}}{z+e^{\imath \alpha}}+\frac{e^{-\imath p}-e^{-\imath x}}{z-1}\bigg)$$
 	
 	and, therefore:
 	
 	\[\begin{split}\varphi_1(z)&=h'(z)+g'(z)\\&=\frac{1}{\pi \imath}\bigg(
 	\frac{\cos x-\cos y}{z-e^{\imath \alpha}}+\frac{\cos y-\cos s}{z+1}+\frac{\cos s- \cos p}{z+e^{\imath \alpha}}+\frac{\cos p-\cos x}{z-1}\bigg)
 	\\&=\frac{1}{\pi \imath}\bigg(
 	\frac{\cos (q-p+2\beta)-\cos p}{z-e^{\imath \alpha}}+\frac{\cos p-\cos (p-q+2\beta)}{z+1}
 	\\ & \ \ +\frac{\cos (p-q+2\beta)- \cos p}{z+e^{\imath \alpha}}+\frac{\cos p-\cos (q-p+2\beta)}{z-1}\bigg)                                                                    \end{split}
 	\]

 	and
 	
 	\[\begin{split}\varphi_2(z)&=-\imath \big(h'(z)-g'(z)\big)
 	\\&=\frac{1}{\pi \imath}\bigg(
 	\frac{\sin x-\sin y}{z-e^{\imath \alpha}}+\frac{\sin y-\sin s}{z+1}+\frac{\sin s- \sin p}{z+e^{\imath \alpha}}+\frac{\sin p-\sin x}{z-1}\bigg)
 	\\&=\frac{1}{\pi \imath}\bigg(
 	\frac{\sin (q-p+2\beta)+\sin p}{z-e^{\imath \alpha}}-\frac{\sin p+\sin (p-q+2\beta)}{z+1}
 	\\&+\frac{\sin (p-q+2\beta)- \sin p}{z+e^{\imath \alpha}}+\frac{\sin p-\sin (q-p+2\beta)}{z-1}\bigg).\end{split}
 	\]
 	Now we set
 	$$u(z) = \Re \int_0^z \varphi_1(\zeta) d\zeta,$$
 	$$v(z)=\Re \int_0^z \varphi_2(\zeta) d\zeta.$$
 	Let $$P(r,\psi-t) = \frac{1}{2\pi}\frac{1-r^2}{1+r^2-2r\cos(\psi-t)}.$$ From the Poisson representation formula
 	
 	\[\begin{split}f(re^{\imath t})&=\int_{0}^{\alpha}P(r,\psi-t)e^{\imath x} d\psi+
 	\int_{\alpha}^{\pi}P(r,\psi-t) d\psi
 	\\&+\int_{\pi}^{\pi+\alpha}P(r,\psi-t)e^{\imath s} d\psi+
 	\int_{\pi+\alpha}^{2\pi}P(r,\psi-t)e^{\imath p} d\psi \end{split}\]
 	we have
 	\[\begin{split}f(0)&=\frac{\alpha(e^{\imath x}+e^{\imath s})+(\pi-\alpha)(e^{\imath y}+e^{\imath p})}{2\pi}
 	\\&=\frac{\alpha e^{2\imath\beta}\cos(p-q)+(\pi-\alpha)\cos p}{\pi}.\end{split}
 	\]
 	To determine $u$ and $v$ we will first evaluate
 	$$ \Im \int_0^z \frac{1}{\zeta-e^{\imath \gamma}} d\zeta,$$
 	for $\gamma \in (0,2\pi).$
 	This is the content of the following claim.
 	\begin{claim}
 		$$\Im \int_0^z \frac{1}{\zeta-e^{\imath \gamma}} d\zeta=\tan^{-1}\frac{r\sin (\gamma- t)}{1-r\cos(\gamma-t)},$$
 		
 		for $z=re^{\imath t}.$
 	\end{claim}
 	\begin{proof} We start from $$	 \frac{1}{\zeta-e^{\imath\gamma}}=\frac{-e^{-\imath\gamma}}{1-e^{-\imath\gamma}\zeta}=-e^{-\imath\gamma}\sum_{n=0}^{+\infty}\big(e^{-\imath\gamma}\zeta\big)^n. $$
 		Integrating this sum on the line segment $0$ to $z$, we get:
 		$$\int_0^z \frac{1}{\zeta-e^{\imath \gamma}} d\zeta=-e^{-\imath\gamma}\sum_{n=0}^{+\infty}e^{-\imath n\gamma}\frac{\zeta^{n+1}}{n+1}=\log\big(1-e^{-\imath\gamma}z\big).$$
 		Taking the imaginary parts, we get the desired conclusion.
 	\end{proof}
 	Using this claim and formulae for $\varphi_1$ and $\varphi_2$,
 	we finally find the closed form for $u$ and $v$ as it is given by \eqref{eq:u} and \eqref{eq:v}.
 	\end{proof}

\subsection{Explicit calculation of $a, b, \theta$}This subsection provides further calculations of previously introduced parameters as functions of $p$ and $q$. These computations constitute an important step toward deriving the final expression for the Gaussian curvature at the point lying above the origin.

From
\begin{equation}
\label{koef}
\begin{split}2\pi  \imath(z^2-1)(z^2&-e^{2\imath\alpha})h'(z)=z^2\big(e^{\imath\alpha}(e^{\imath x}-e^{\imath y})-(e^{\imath y}-e^{\imath s})-e^{\imath\alpha}(e^{\imath s}-e^{\imath p})+e^{\imath p}-e^{\imath x})\big)
\\&-z\big(e^{\imath x}-e^{\imath y}+e^{2\imath\alpha}(e^{\imath y}-e^{\imath s})+e^{\imath s}-e^{\imath p}+e^{2\imath\alpha}(e^{\imath p}-e^{\imath x})\big)
\\&-e^{\imath \alpha}(e^{\imath x}-e^{\imath y})+e^{2\imath\alpha}(e^{\imath y}-e^{\imath s})+e^{\imath \alpha}(e^{\imath s}-e^{\imath p})-e^{2\imath\alpha}(e^{\imath p}-e^{\imath x})
\end{split}
\end{equation}
and the expression for $h'$ we have:
\[\begin{split}2\pi \imath b&=e^{\imath \alpha}\bigg[e^{\imath s}-e^{\imath p}-e^{\imath x}+e^{\imath y}+e^{\imath\alpha}(e^{\imath y}+e^{\imath x}-e^{\imath p}-e^{\imath s})\bigg]
\\&=2\imath e^{\imath \alpha}\bigg[e^{2\imath\beta}\sin(p-q)-\sin p-e^{\imath\alpha}\bigg(\sin p +e^{2\imath\beta}\sin(p-q)\bigg)\bigg]
\\&=-4\imath e^{\frac{3\imath\alpha}{2}}\bigg[\sin p\cos\frac{\alpha}{2}-\sin(p-q)\sin(2\beta)\sin\frac{\alpha}{2}+\imath \cos(2\beta)\sin(p-q)\sin\frac{\alpha}{2}\bigg]
\end{split}
\]
and thus
$$b=\frac{e^{\imath\alpha}}{\pi}\bigg(e^{2\imath\beta}\sin(p-q)-\sin p-e^{\imath\alpha}\big(\sin p +e^{2\imath\beta}\sin(p-q)\big)\bigg),$$
i.e. \begin{equation}\label{formula}b=-\frac{e^{\imath \alpha } \left(\left(1+e^{\imath \alpha }\right) \sin p+e^{2 \imath \beta } \left(-1+e^{\imath \alpha }\right) \sin(p-q)\right)}{\pi }.\end{equation}
From \eqref{formula} we have
\begin{equation}\label{bsquare}\begin{split}|b|^2&=\frac{4}{\pi^2}\bigg(\big(\sin p\cos\frac{\alpha}{2}-\sin(p-q)\sin(2\beta)\sin\frac{\alpha}{2}\big)^2+\cos^2(2\beta)\sin^2(p-q)\sin^2\frac{\alpha}{2}\bigg)
\\&= \frac{2}{\pi^2}\bigg(2\sin^2(q-p)\sin^2\frac{\alpha}{2}+2\sin^2p\cos^2\frac{\alpha}{2}+2\sin p\sin(q-p)\sin(2\beta)\sin\alpha\bigg).                                        \end{split}\end{equation}
Inserting
\begin{equation}
\label{eq:sinalpha}
\sin\alpha=\frac{2\sqrt{\sin p\sin(q-p)}}{\sin p+\sin(q-p)}
\end{equation}
and
\begin{equation}
\label{eq:sinbeta}
\sin(2\beta)=\frac{\sin p+\sin(q-p)}{1+\sin(q-p)\sin p}
\end{equation}
in \eqref{bsquare}, we get
\[\begin{split}\frac{\pi^2}{2}|b|^2&= \frac{2\sin p\sin^2(q-p)}{\sin(q-p)+\sin p}+\frac{2\sin^2 p\sin(q-p)}{\sin(q-p)+\sin p}
\\&+2\sin p\sin(q-p)\frac{\sqrt{4\sin p\sin(q-p)}}{\sin p+\sin(q-p)}\frac{\sin p+\sin(q-p)}{1+\sin(q-p)\sin p}
\\&=  \frac{2\big(1+\sin(q-p)\sin p\big)\sin(q-p)\sin p+4\sin(q-p)\sin p\sqrt{\sin(q-p)\sin p}}{1+\sin(q-p)\sin p}                                                               \\&=\frac{2\big(\sin(q-p)\sin p+\sqrt{\sin(q-p)\sin p}\big)^2}{1+\sin(q-p)\sin p},\end{split}\]
i.e.
\begin{equation}
\label{eq:bsquare2}
|b|^2=\frac{4\big(\sin(q-p)\sin p+\sqrt{\sin(q-p)\sin p}\big)^2}{\pi^2(1+\sin(q-p)\sin p)}.
\end{equation}
The formula \eqref{koef} gives also a way to find $a$. Indeed, the linear coefficient is equal to $-2\overline{a}b.$ Hence, we have
\[\begin{split}-2\overline{a}b&=-\frac{1}{2\pi \imath}\big(1-e^{2\imath\alpha}\big)\big(e^{\imath x}+e^{\imath s}-e^{\imath p}-e^{\imath y}\big)
\\&= -\frac{1}{\pi}\big(1-e^{2\imath\alpha}\big)\big(e^{2\imath\beta}\cos(p-q)-\cos p\big).
\end{split}
\]
From formula \eqref{formula} we get
$$\overline{a}= \sin\alpha \cdot \frac{\cos p- e^{2\imath\beta}\cos(q-p)}{e^{2\imath\beta}\sin(p-q)(1-e^{\imath\alpha})-\sin p(1+e^{\imath\alpha})} .$$
By using  \eqref{eq:kosalpha}, \eqref{eq:sinalpha}, \eqref{eq:sinbeta} and $$\cos(2\beta)=\frac{\cos p\cos(q-p)}{1+\sin(q-p)\sin p},$$ after a long, but straightforward calculation, we get:

\begin{equation}
\label{eq:a}
a=\frac{\cos(q-p)-\cos p-\imath \sin q}{1-\cos q+2\sqrt{\sin p\sin(q-p)}+\imath\big(\sin(q-p)-\sin p\big)}.
\end{equation}

Some elementary identities
lead us to

\begin{equation}
\label{eq:modula}
|a|=\sqrt{\frac{1-\sqrt{\sin p\sin(q-p)}}{1+\sqrt{\sin p\sin(q-p)}}},
\end{equation}
thus giving a more suitable form of $a.$ In fact, by previous observations and some identities for trigonometric functions, we have:
\begin{equation}
\label{eq:abolje}
a=\sqrt{\frac{1-\sqrt{\sin p\sin(q-p)}}{1+\sqrt{\sin p\sin(q-p)}}}\big(\cos\delta+ \imath \sin\delta\big),
\end{equation}
where
\begin{equation}\begin{split}\label{eq:kosargumenta}&\cos\delta=
-\frac{\sin(\frac{q}{2}-p)}{\sin\frac{q}{2}\sqrt{1-\sin(q-p)\sin p}}\end{split}
\end{equation}
and
\begin{equation}\begin{split}\label{eq:sinargumenta}&\sin\delta
=-\frac{\cos\frac{q}{2}\sqrt{\sin(q-p)\sin p}}{\sin\frac{q}{2}\sqrt{1-\sin(q-p)\sin p}}.\end{split}
\end{equation}
One easily observe that the last two formulas are not well defined for $\sin p\sin(q-p)=1,$ i.e. $p=\frac{\pi}{2}, q=\pi.$ From now on, we will exclude this (trivial) case and assume $(p,q)\neq (\frac{\pi}{2},\pi).$\\

From \eqref{dila}  we obtain the following formula
\begin{equation}\label{hgzero}\begin{split}\frac{g'(0)}{h'(0)}&=\mathbf q^2(0)\\&=-e^{-\imath(2\beta+q)}M^2(0)\\&=\frac{e^{-2 \imath \beta } \left(-e^{2 \imath\beta } \left(1+e^{\imath \alpha }\right) \sin p-\left(-1+e^{\imath \alpha }\right) \sin(p-q)\right)}{\left(1+e^{\imath \alpha }\right) \sin p+e^{2 \imath \beta } \left(-1+e^{\imath \alpha }\right) \sin(p-q)}.\end{split}\end{equation}
Also, we have that
\begin{equation}
\label{eq:varphi3}
\varphi_3=-2\imath\mathbf{p}\mathbf{q}=-2\imath e^{\imath\theta}\frac{b(z-a)(1-z\overline{a})}{(z^2-1)(z^2-e^{2\imath\alpha})}.
\end{equation}
Now from \eqref{hgzero}, the formula $\mathbf{q}(0)=-ae^{\imath \theta}$ and previous relations we get  $$\theta =  \tan^{-1}[\cos(p-q) \tan(p)].$$
The third coordinate of our Scherk type surface is explicitly stated as a function depending on parameters $p$ and $q$: \begin{equation}\label{eq:T}\begin{split}T(z)&=\Re \int_0^z\varphi_3(\zeta)d\zeta=\Re\int_0^z(-2\imath \mathbf{p}(\zeta)\mathbf{q}(\zeta)d\zeta\\&=\Im\bigg[\frac{b e^{-\imath (\alpha -\theta )}}{e^{- \imath \alpha }-e^{ \imath \alpha }} \bigg(\left(1+|a|^2\right) \log[\frac{1-z^2}{1-e^{-2\imath \alpha } z^2}]\\&+4 \Re (a) \tanh^{-1}(z)-4\Re(ae^{-\imath \alpha})\tanh^{-1}\left[e^{-\imath \alpha } z\right] \bigg)\bigg].\end{split}\end{equation}

Based on \eqref{eq:u}, \eqref{eq:v} and \eqref{eq:T} we obtain a Scherk type surface $\varpi(z) = (u(z), v(z) , T(z))$ in Figure~\ref{21}.
\begin{figure}
\centering
\includegraphics{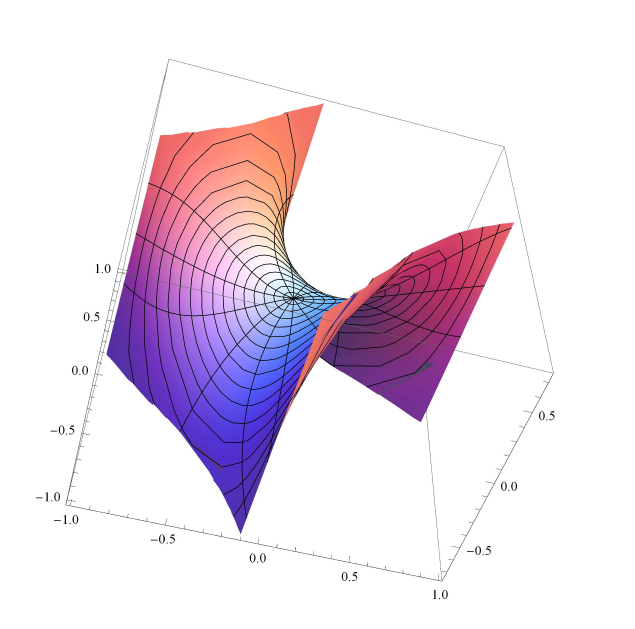}
\caption{A Scherk type surface for $p=\pi/2+0.1$, $q=\pi-0.1$. }
\label{21}
\end{figure}
over the quadrilateral shown in Figure~\ref{22}
\begin{figure}
\centering
\includegraphics{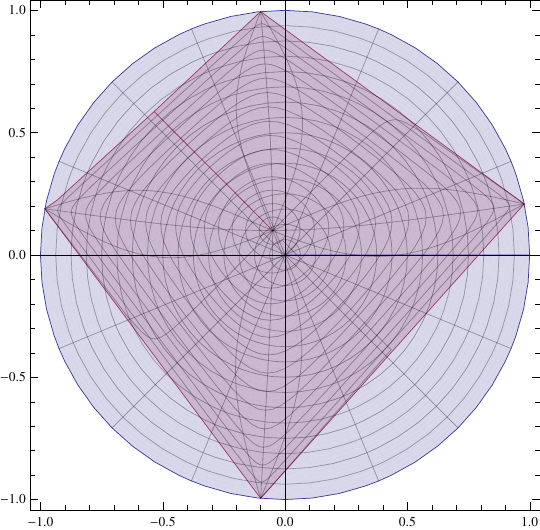}
\caption{A bicentric quadrilateral inscribed in the unit disk for $p=\pi/2+0.1$, $q=\pi-0.1$. }
\label{22}
\end{figure}
\section{Curvature of two-parameter Scherk type surfaces and the proof of Theorem~\ref{seconda}}\label{sec3}
After we gave the construction of two-parameter family of Scherk type surfaces, we continue with calculation of the curvature. Let us recall the formula for the curvature
\begin{equation*}
\label{eq:krivina1}
 \mathcal{K}=-\frac{4|\mathbf{q}'|^2}{|\mathbf{p}|^2(1+|\mathbf{q}|^2)^4}.
\end{equation*}
By using $|\mathbf{q}'|=\frac{1-|a|^2}{|1-\overline{a}z|^2}$ and formulae  \eqref{eq:pbold}, \eqref{eq:qbold} and \eqref{eq:modula}, we find that the Gaussian curvature at the point over the $0=f(z_\circ)$ is equal to
\begin{equation}
\label{eq:krivina2}
\mathcal{K} = -\frac{4(1-|a|^2)^2}{|b|^2} \frac{|z_\circ^2-1|^2 |z_\circ^2-e^{2\imath\alpha}|^2}{(|1-z_\circ\overline{a}|^2+|z_\circ-a|^2)^4}.
\end{equation}
Using already determined values of $a, b$ we get
\begin{equation}\label{eq:krivina3}\begin{split}-\mathcal{K}&=\frac{4 \pi^2(1+\sin p\sin(q-p))}{\left(1+\sqrt{\sin p\sin(q-p)}\right)^4} \frac{|1-z_\circ^2|^2 |1-z_\circ^2 e^{2\imath\alpha}|^2}{((1+|z_\circ|^2)(1+|a|^2) - 4 \Re (a\bar z_\circ))^4}
\\&=\frac{\pi^2}{4}\big(1+\sin p\sin(q-p)\big)\frac{|1-z_\circ^2|^2 |1-z_\circ^2 e^{2\imath\alpha}|^2}{((1+|z_\circ|^2) - 4 \frac{\Re (a\bar z_\circ)}{(1+|a|^2)})^4}.\end{split}
\end{equation}

The problem now reduces to maximizing the final expression over $(p, q) \in \mathcal{S}$. Note that the quantities $z_{\circ}$, $\alpha$, and $a$ also depend on $p$ and $q$. To establish the main inequality for the Gaussian curvature, our next objective is to localize the zero of the mapping $f$.

It turns out that it is sufficient to determine in which of the four sectors

$\{z : \arg z \in (0,\alpha)\}$ ,$ \{z : \arg z \in (\alpha,\pi)\}$, $ \{z : \arg z \in (\pi,\pi+\alpha)\},$ or $\{z : \arg z \in (\pi+\alpha,2\pi)\}$
(or on their boundaries) the zero of $f$ lies.

By analyzing the images of these sectors under $f$, we identify which boundary encloses the origin, or possibly passes through it. For instance, we conclude that the sector $\{z : \arg z \in (0,\alpha)\}$ contains $z_{\circ}$ if the image of its boundary under $f$ has a tangent vector whose argument varies monotonically. If $z_{\circ}$ lies on one of the rays $\arg z = \beta$, with $\beta \in \{0, \alpha, \pi, \pi+\alpha\}$, we determine this using continuity arguments.

A precise formulation of this preliminary observation is given in Lemma~\ref{mainlemma}.

Let us show that every sector is mapped by $f$ onto a curvilinear quadrilateral. Observe that the boundary function $G(t)=F(e^{\imath t})$ of the harmonic mapping is piecewise continuous but has four jump discontinuities: $(t_1,t_2,t_3,t_4)=(0,\alpha, \pi, \pi+\alpha)$, $t_5=t_1+2\pi$. Moreover $G(0-)=e^{\imath p}$, $G(0+)=e^{\imath x}$, $G(\alpha-)=e^{\imath x}$, $G(\alpha+)=e^{\imath y}$, $G(\pi-)=e^{\imath y}$, $G(\pi+)=e^{\imath s}$, $G(\pi+\alpha-)=e^{\imath s}$, $G(\pi+\alpha+)=e^{\imath p}$.

Then $$B_k=\lim_{r\to 1}f(r e^{\imath t_k})=\frac{1}{2}\left(G(t_k-)+G(t_k+)\right), k=1,2,3,4,$$ and put $B_5=B_ 1$.
Let $A_1=e^{\imath p}$, $A_2=e^{\imath x}$, $A_3=e^{\imath y}$, $A_4=e^{\imath s}$, $A_5=A_1$. This is a basic fact but we also refer to \cite[p.~13]{DUREN}.

The boundary of the sector $\arg z \in (t_k, t_{k+1})$, for $k = 1, 2, 3, 4$, consists of two curves $\gamma_k$ and $\gamma_{k+1}$ (with $\gamma_5 := \gamma_1$), both emanating from the common point $f(0)$ and terminating at the points $B_k$ and $B_{k+1}$, respectively. These are connected by linear segments $B_kA_k$ and $A_kB_{k+1}$, forming a closed contour. The image of the sector under the mapping $f$ is a domain denoted by $Q_k$, which is bounded by the curves $\gamma_k$, $\gamma_{k+1}$, and the line segments $B_kA_k$ and $A_kB_{k+1}$.

The quadrilateral $Q(e^{\imath p}, e^{\imath x}, e^{\imath y}, e^{\imath s})$ can thus be expressed as the union of these four domains:
\[
Q(e^{\imath p}, e^{\imath x}, e^{\imath y}, e^{\imath s}) = \bigcup_{k=1}^4 \overline{Q_k} \quad \text{(see Figure~\ref{quad})}.
\]

Our goal is to determine which of the domains \( Q_k \), for \( k = 1, 2, 3, 4 \), or their boundaries, contains a zero of the function \( f \). To this end, we examine the curves
\[
\delta_k = \gamma_k + \gamma_{k+1}^{-}, \quad k = 1, 2, 3, 4,
\]
where the superscript "$-$" denotes the reverse orientation of the curve.

A curve in the codomain is said to encircle the origin if the angle of its tangent vector varies monotonically-either increasing or decreasing-indicating that it winds around the origin in a positive or negative sense, respectively. To identify which curve \( \delta_k \) encircles zero, we determine which of these curves has a tangent vector field that changes direction monotonically.

For a fixed pair \( (p,q) \neq \left( \frac{\pi}{2}, \pi \right) \), there is exactly one sector, or the segment between two adjacent sectors, that contains the zero \( z_\circ \).

We will determine it in the next lemma.
\begin{lemma}[Localization of the zero $z_\circ$]\label{mainlemma}
Assume that $(p,q)$ is a pair of real numbers that belongs to the domain $\mathcal{S}$ defined in \eqref{QQ}  and assume that $z_\circ$ is the zero of $f$. Then we have
\begin{enumerate}
\item[a)] $\arg z_\circ \in (\alpha,\pi),$ for $q<\pi, q< 2p,$
\item[b)] $\arg z_\circ \in (\pi+\alpha,2\pi),$ for $q>\pi, q> 2p,$
\item[c)] $\arg z_\circ \in (\pi,\pi+\alpha),$ for $q \in (\pi,2p),$
\item[d)] $\arg z_\circ \in (0,\alpha)$, for $q \in (2p,\pi).$
\end{enumerate}
By using continuity of $f,$ we also conclude:
\begin{enumerate}
\item[$a')$] $\arg z_\circ=0,$ for $q=\pi$ and $p<\frac{\pi}{2},$
\item[$b')$] $\arg z_\circ=\alpha,$ for $q=2p$ and $p<\frac{\pi}{2},$
\item[$c')$] $\arg z_\circ=\pi,$ for $q=\pi$ and $p>\frac{\pi}{2},$
\item[$d')$] $\arg z_\circ=\pi+\alpha$, for $q=2p$ and $p>\frac{\pi}{2}.$
\end{enumerate}
\end{lemma}
The only case that is not considered by $a)-d)$ or $a')-d')$ is $p=\frac{\pi}{2}$ and $q=\pi,$ when $\alpha=\frac{\pi}{2},$ $z_{\circ}=0$ and $\big|\mathcal{K}\big|=\frac{\pi^2}{2},$ which makes this case trivial. The proof of Lemma~\ref{mainlemma} is presented in the following four subsections with an easy argument of the concluding the cases $a')-d')$ from them.
\begin{figure}
\centering
\includegraphics{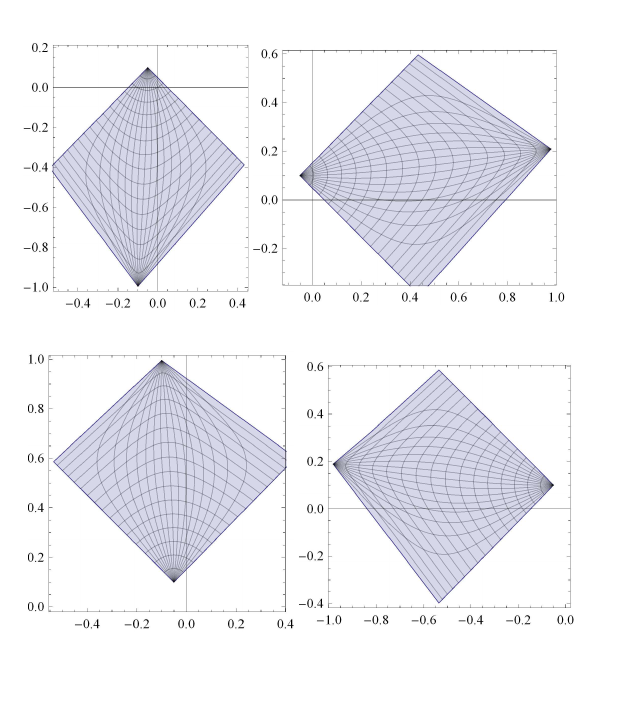}
\caption{The images of four consecutive sectors, $\arg z\in (\alpha,\pi)$, $(\pi, \pi+\alpha)$, $(\pi+\alpha,2\pi)$ and $(0,\alpha)$ under the harmonic mapping $f$.
In this case $p=\pi/2+0.1$, $q=\pi-0.1$, so $q<\min\{2p,\pi\}$. Notice that the first curvilinear quadrilateral contains the zero, and we are in the case a).}
\label{quad}
\end{figure}
\subsubsection{The case a)} From the formulae for $u$ and $v$ (\eqref{eq:u}, \eqref{eq:v}), we find:
\[\begin{split}u(re^{\imath\alpha})&=\frac{\cos p-\cos(p-q+2\beta)}{\pi}\tan^{-1}\frac{r\sin\alpha}{1+r\cos\alpha}
\\&+
\frac{\cos(q-p+2\beta)-\cos p}{\pi}\tan^{-1}\frac{r\sin\alpha}{1-r\cos\alpha}\end{split}\]
and
\[\begin{split}v(re^{\imath\alpha})&=-\frac{\sin p+\sin(p-q+2\beta)}{\pi}\tan^{-1}\frac{r\sin\alpha}{1+r\cos\alpha}
\\&+
\frac{\sin(q-p+2\beta)-\sin p}{\pi}\tan^{-1}\frac{r\sin\alpha}{1-r\cos\alpha},\end{split}\]
while
\[\begin{split}u(-r)&=\frac{\cos p-\cos(q-p+2\beta)}{\pi}\tan^{-1}\frac{r\sin\alpha}{1+r\cos\alpha}
\\&+
\frac{\cos(p-q+2\beta)-\cos p}{\pi}\tan^{-1}\frac{r\sin\alpha}{1-r\cos\alpha}\end{split}\]
and
\[\begin{split}v(-r)&=-\frac{\sin p+\sin(q-p+2\beta)}{\pi}\tan^{-1}\frac{r\sin\alpha}{1+r\cos\alpha}\\&+
\frac{\sin(p-q+2\beta)-\sin p}{\pi}\tan^{-1}\frac{r\sin\alpha}{1-r\cos\alpha}.\end{split}\]
Now, we consider the following quantity:
\[\begin{split}\Psi(r,\varphi)&=\arg \bigg( \frac{\partial f(re^{\imath \varphi})}{\partial r} \bigg)
      \\&=\arg \bigg(\frac{\partial u(re^{\imath \varphi})}{\partial r}+ \imath \frac{\partial v(re^{\imath \varphi})}{\partial r} \bigg)
       \\&= \tan^{-1} \bigg(\frac{\frac{\partial v(re^{\imath \varphi})}{\partial r}}{\frac{\partial u(re^{\imath \varphi})}{\partial r}}\bigg).
       \end{split}
       \]
The monotonicity of the argument provides an important property of the tangent vector to the curve \( f(re^{\imath\varphi}) \), where \( \varphi \) is fixed. Specifically, the monotonic behavior of the argument of the tangent vector indicates that the curve winds around the origin in one of two possible orientations: counterclockwise if the argument is increasing, and clockwise if it is decreasing. Further, we have
\[\begin{split}\frac{\partial }{\partial r}u(re^{\imath\alpha})&=\frac{\cos p-\cos(p-q+2\beta)}{\pi}\frac{\sin\alpha}{1+r^2+2r\cos\alpha}
\\&+
\frac{\cos(q-p+2\beta)-\cos p}{\pi}\frac{\sin\alpha}{1+r^2-2r\cos\alpha},\end{split}\]
\[\begin{split}\frac{\partial }{\partial r}v(re^{\imath\alpha})&=-\frac{\sin p+\sin(p-q+2\beta)}{\pi}\frac{\sin\alpha}{1+r^2+2r\cos\alpha}
\\&+
\frac{\sin(q-p+2\beta)-\sin p}{\pi}\frac{\sin\alpha}{1+r^2-2r\cos\alpha}.\end{split}\]
The function $\Psi$ has the same monotonicity as the function $\tan\Psi,$ hence we will find the derivative of
$$\frac{\frac{\partial v(re^{\imath \varphi})}{\partial r}}{\frac{\partial u(re^{\imath \varphi})}{\partial r}}=\frac{P}{Q},$$
where  $$P=A(1+r^2+2r\cos\alpha)+B(1+r^2-2r\cos\alpha),$$ $$Q=C(1+r^2+2r\cos\alpha)+D(1+r^2-2r\cos\alpha),$$
with:
$$A=  \sin(q-p+2\beta)-\sin p=2\sin(\frac{q}{2}-p+\beta)\cos(\beta+\frac{q}{2}),$$
$$B=-\sin p-\sin(p-q+2\beta)=-2\sin(p+\beta-\frac{q}{2})\cos(\frac{q}{2}-\beta), $$
$$C=\cos(q-p+2\beta)-\cos p=-2\sin(\beta+\frac{q}{2})\sin(\frac{q}{2}+\beta-p),  $$
$$ D=\cos p-\cos(p-q+2\beta)=-2\sin(p+\beta-\frac{q}{2})\sin(\beta-\frac{q}{2}).$$
The straightforward calculation gives
\[\begin{split}&P'Q-Q'P\\&=\big(2A(r+\cos\alpha)+2B(r-\cos\alpha)\big)\bigg(C(1+r^2+2r\cos\alpha)+D(1+r^2-2r\cos\alpha)\bigg)
\\&-\big(2C(r+\cos\alpha)+2D(r-\cos\alpha)\big)\bigg(A(1+r^2+2r\cos\alpha)+B(1+r^2-2r\cos\alpha)\bigg)
\\&=4(1-r^2)(AD-BC)\cos\alpha.\end{split}\]
For this range of $p$ and $q,$ we have $\cos\alpha=\frac{\tan(\frac{q}{2}-p)}{\tan\frac{q}{2}} \leq 0,$ since $0\leq \frac{q}{2}<\frac{\pi}{2}$ and $-\frac{\pi}{2}<\frac{q}{2}-p< 0.$ Also,
\[\begin{split}AD-BC&=4\sin(\beta+\frac{q}{2}+\beta-p)
\sin(\beta+p-\frac{q}{2})\\&\times \bigg(-\sin(\beta-\frac{q}{2})\cos(\beta+\frac{q}{2})-\cos(\beta-\frac{q}{2})\sin(\beta+\frac{q}{2})\bigg)
\\&=-4\sin(\beta+\frac{q}{2}-p)\sin(\beta+p-\frac{q}{2})\sin(2\beta)< 0,\end{split}\]
because $\beta+\frac{q}{2}-p,\beta+p-\frac{q}{2} \in (0,\pi)$ and $\beta \in (0,\frac{\pi}{2}).$

We conclude $P'Q-Q'P> 0$, i.e. the tangent vector at $f(re^{\imath \alpha})$ has an increasing argument. (This follows from \eqref{eq:poredak}; it will be also important for the other cases.)

For the second part of the curve, we have:
\[\begin{split}\frac{\partial }{\partial r}u(-r)&=\frac{\cos p-\cos(q-p+2\beta)}{\pi}\frac{\sin\alpha}{1+r^2+2r\cos\alpha}
\\&+
\frac{\cos(p-q+2\beta)-\cos p}{\pi}\frac{\sin\alpha}{1+r^2-2r\cos\alpha},\end{split}\]
\[\begin{split}\frac{\partial }{\partial r}v(-r)&=-\frac{\sin p+\sin(q-p+2\beta)}{\pi}\frac{\sin\alpha}{1+r^2+2r\cos\alpha}
\\&+
\frac{\sin(p-q+2\beta)-\sin p}{\pi}\frac{\sin\alpha}{1+r^2-2r\cos\alpha}.\end{split}\]

This time, we consider  $$\tan\Psi=\frac{\frac{\partial }{\partial r}v(-r)}{\frac{\partial }{\partial r}u(-r)}=\frac{P}{Q}=\frac{A(1+r^2+2r\cos\alpha)+B(1+r^2-2r\cos\alpha}{C(1+r^2+2r\cos\alpha)+D(1+r^2-2r\cos\alpha)},$$ with:
$$A=\sin(p-q+2\beta)-\sin p=2\sin(\beta-\frac{q}{2})\cos(p+\beta-\frac{q}{2}),$$
$$B=-\sin p-\sin(q-p+2\beta)=-2\sin(\beta+\frac{q}{2})\cos(p-\beta-\frac{q}{2}),$$
$$C=\cos(p-q+2\beta)-\cos p=-2\sin(p+\beta-\frac{q}{2})\sin(\beta-\frac{q}{2}),$$
$$D=\cos p-\cos(q-p+2\beta)=-2\sin(\beta+\frac{q}{2})\sin(p-\beta-\frac{q}{2}).$$
Since we work on the same range of $p$'s and $q$'s, $\cos\alpha<0.$ Here we have:
\[\begin{split}AD-BC
&=-4\sin(\beta+\frac{q}{2})\sin(\beta-\frac{q}{2})\\& \times\bigg(\sin(p-\beta-\frac{q}{2})\cos(p+\beta-\frac{q}{2})+\cos(p-\beta-\frac{q}{2})\sin(p+\beta-\frac{q}{2})\bigg)
\\&=-4\sin(\beta+\frac{q}{2})\sin(\beta-\frac{q}{2})\sin(2p-q)>0,\end{split}\]
because $\frac{q}{2}-\beta \in (0,p), \beta+\frac{q}{2} \in (p,\pi)$ and $2p-q \in (0,\pi).$
This implies $P'Q-Q'P=4(1-r^2)(AD-BC)\cos\alpha<0.$

Finally, we can conclude that the argument of the tangent vector of $f(re^{\imath\varphi})$ is increasing for $\varphi=\alpha$, while it is decreasing for $\alpha=\pi,$ considered as a function on $r.$ This exactly means that the whole curve, which consists of these two parts, which starts at $f(-1),$ goes as $f(-r)$ till $f(0)$ and then continues as $f(re^{\imath\alpha})$ till the end in $f(e^{\imath\alpha})$ has the tangent vector with an increasing argument i.e. it goes around the $0.$ This implies that $\arg z_\circ \in (\alpha,\pi).$
\subsubsection{The case b)} Let $q>\max\{\pi, 2p\}.$ Here we consider $f(re^{\imath\varphi})$ with $\varphi=0$ and $\varphi=\pi+\alpha.$ We easily find that
\[\begin{split}\frac{\partial}{\partial r}u(r)&=\frac{\cos(q-p+2\beta)-\cos p}{\pi}\frac{\sin\alpha}{1+r^2-2r\cos\alpha}
\\&+\frac{\cos p-\cos(p-q+2\beta)}{\pi}\frac{\sin\alpha}{1+r^2+2r\cos\alpha},\end{split}\]

\[\begin{split}\frac{\partial}{\partial r}v(r)&=\frac{\sin(q-p+2\beta)+\sin p}{\pi}\frac{\sin\alpha}{1+r^2-2r\cos\alpha}
\\&+\frac{\sin p-\sin(p-q+2\beta)}{\pi}\frac{\sin\alpha}{1+r^2+2r\cos\alpha}\end{split}\]
and in the expression for $$\tan\Psi=\frac{P}{Q}=\frac{A(1+r^2+2r\cos\alpha)+B(1+r^2-2r\cos\alpha)}{C(1+r^2+2r\cos\alpha)+D(1+r^2-2r\cos\alpha)},$$ we get:
$$A=\sin(q-p+2\beta)+\sin p=2\sin(\beta+\frac{q}{2})\cos(\beta+\frac{q}{2}-p),$$
$$B=\sin p-\sin(p-q+2\beta)=2\sin(\frac{q}{2}-\beta)\cos(p+\beta-\frac{q}{2}),$$
$$C=\cos(q-p+2\beta)-\cos p=-2\sin(\beta+\frac{q}{2})\sin(\frac{q}{2}+\beta-p),$$
$$D=\cos p-\cos(p-q+2\beta)=2\sin(\beta-\frac{q}{2})\sin(p+\beta-\frac{q}{2}).$$
For $q>\pi, q>2p$ we have $\cos\alpha=\frac{\tan(\frac{q}{2}-p)}{\tan\frac{q}{2}}<0,$ while:
\[\begin{split}AD-BC&=4\sin(\beta+\frac{q}{2})\sin(\beta-\frac{q}{2})\\&\times \bigg(\sin(p+\beta-
\frac{q}{2})\cos(\beta+\frac{q}{2}-p)-\sin(\frac{q}{2}+\beta-p)\cos(p+\beta-\frac{q}{2})\bigg)
\\&=4\sin(\beta+\frac{q}{2})\sin(\beta-\frac{q}{2})\sin(2p-q)>0.\end{split}\]
Therefore, $P'Q-Q'P=4(1-r^2)(AD-BC)\cos\alpha<0.$

Similarly,
$$\frac{\partial}{\partial r}u(re^{\imath(\pi+\alpha)})=\frac{\cos(p-q+2\beta)-\cos p}{\pi}\frac{\sin\alpha}{1+r^2-2r\cos\alpha}$$
$$+\frac{\cos p-\cos(q-p+2\beta)}{\pi}\frac{\sin\alpha}{1+r^2+2r\cos\alpha},$$
$$\frac{\partial}{\partial r}v(re^{\imath(\pi+\alpha)})=\frac{\sin(p-q+2\beta)+\sin p}{\pi}\frac{\sin\alpha}{1+r^2-2r\cos\alpha}$$
$$+\frac{\sin p-\sin(q-p+2\beta)}{\pi}\frac{\sin\alpha}{1+r^2+2r\cos\alpha}$$
and for $\frac{P}{Q}$ we get the analogous expression, but with different values of the coefficients $(A,B,C,D)$:
$$A=\sin(p-q+2\beta)+\sin p=2\sin(p+\beta-\frac{q}{2})\cos(\frac{q}{2}-\beta),$$
$$B=\sin p-\sin(q-p+2\beta)=2\sin(p-\frac{q}{2}-\beta)\cos(\frac{q}{2}+\beta),$$
$$C=\cos(p-q+2\beta)-\cos p=-2\sin(p+\beta-\frac{q}{2})\sin(\beta-\frac{q}{2}),$$
$$D=\cos p-\cos(q-p+2\beta)=-2\sin(\beta+\frac{q}{2})\sin(p-\beta-\frac{q}{2}).$$
For this part of curve, we get:
\[\begin{split}AD-BC&=-4\sin(p-\beta-\frac{q}{2})\sin(p+\beta-\frac{q}{2})\\ &\times \bigg(-\sin(\beta+\frac{q}{2})\cos(\beta-\frac{q}{2})+\cos(\beta+\frac{q}{2})\sin(\beta-\frac{q}{2})\bigg)
\\&=-4\sin(p-\beta-\frac{q}{2})\sin(p+\beta-\frac{q}{2})\sin q<0,\end{split}\]
thus $P'Q-Q'P>0.$

The above calculations in this case allow us to conclude that the tangent vector of the curve, starting from $f(e^{\imath(\pi+\alpha)})$ and moving along $f(re^{\imath(\pi+\alpha)})$ through $f(0)$ and then along $f(r)$ ending in $f(1)$ has a decreasing argument, therefore $\arg z_\circ \in (\pi+\alpha, 2\pi).$
\subsubsection{The case c), i.e.  the case $q \in (\pi, 2p)$} Proof of this case is similar. However, we will prove it using the earlier calculations. Namely
in the previous cases, we already computed the coefficients $A, B, C, D$ for all $f(re^{\imath\varphi}), $ with $\varphi=0,\alpha,\pi,\pi+\alpha.$

For $\varphi=\pi,$ we have $AD-BC=-4\sin(\beta+\frac{q}{2})\sin(\beta-\frac{q}{2})\sin(2p-q)>0$ and $\cos\alpha>0,$ giving $P'Q-Q'P>0.$

On the other side, for $\varphi=\pi+\alpha,$ the same quantities are $AD-BC=-4\sin(p+\beta-\frac{q}{2})\sin(p-\beta-\frac{q}{2})\sin q<0$ and $\cos\alpha>0,$ which implies $P'Q-Q'P< 0.$

We conclude now that the tangent vector of the curve starting from $f(-1)$ and moving along $f(-r)$ till $f(0)$ and then along  $f(re^{\imath(\pi+\alpha)})$ till $f(e^{\imath(\pi+\alpha)})$ has a decreasing argument, hence $\arg z_\circ \in (\pi,\pi+\alpha).$

\subsubsection{The case d), i.e.  $q \in (2p,\pi)$}
For $\varphi=0,$ we get $AD-BC=4\sin(\beta+\frac{q}{2})\sin(\beta-\frac{q}{2})\sin(2p-q)>0 $ and $\cos\alpha>0, $ which gives $P'Q-Q'P>0.$ The second part, i.e. for $\varphi=\alpha,$ $AD-BC=-4\sin(2\beta)\sin(\beta+\frac{q}{2}-p)\sin(\beta+p-\frac{q}{2})<0$ implies $P'Q-Q'P<0.$

The curve with starting point $f(1)$, crossing through the point $f(0)$ and ending in $f(e^{\imath\alpha})$ going along $f(r)$ and $f(re^{\imath\alpha})$, respectively, has an increasing argument. Thus $\arg z_\circ \in (0, \alpha).$

By using continuity and limiting process we can conclude the same for $q=2p$ or $q=\pi.$ For example, if $q=\pi$ and $p<\frac{\pi}{2},$ we have $q=\pi>2p,$ which holds in some open neighborhood $\mathcal{O}$ of $(p,q);$ therefore $\mathcal{O}$ has non-empty intersection with the sets $q>\pi, q>2p$ and $q<\pi, q>2p.$ For those sets $\arg z_{\circ} \in (\pi+\alpha, 2\pi)$ or $\arg z_{\circ} \in (0, \alpha),$ respectively, and their joint boundary satisfies $\arg z_{\circ}=0,$ thus giving the argument of the zero $z_{\circ}$ in this case. The same reasoning applies for $b')-d')$.

The same conclusion can be drawn from the expression in \eqref{eq:skalprod}, where one of the summands inside the brackets is zero and the other is non-negative.

Observe that when $q = 2p$, we have $e^{\imath p}e^{\imath s} = e^{\imath x}e^{\imath y}$, implying that the image quadrilateral is a trapezoid. Similarly, in the case $q = \pi$, the relation $e^{\imath p}e^{\imath x} = e^{\imath y}e^{\imath s}$ also leads to a trapezoidal image.

The case in which the image is a deltoid was mentioned in \cite{NITSCHE2}, although no proof was provided. In our setting, this corresponds to the condition $q - p = \frac{\pi}{2}$.

\begin{proof}[Proof of Theorem~\ref{seconda}]  Given the Scherk type minimal surface  $S:(u,v,\mathbf{f}(u,v))$ let $\tilde f$ be the projection of conformal parametrization (i.e. of  Enneper--Weierstrass representation of the minimal graph
$\varpi=(u,v,T):\mathbb{D} \to S\subset \mathbb{D}\times \R$). Then $\tilde f= (u,v)$ is a harmonic mapping of the unit disk onto $Q(a_0,a_1,a_2,a_3)$, which is certainly induced by a step mapping of the unit disk onto the vertices of the quadrilateral. There are four distinct points $b_0, b_1, b_2$, and $b_3$ on the unit circle whose cluster sets under $\tilde{f}$ are the line segments $[a_0, a_1]$, $[a_1, a_2]$, $[a_2, a_3]$, and $[a_3, a_0]$, respectively.
 Then there exists a  unique M\"obius transform $M$, which maps the points  $1,e^{\imath\alpha}, -1, -e^{\imath\alpha}$ onto the points $b_0,b_1,b_2$ and $b_3$. Then the mapping $ f = \tilde f\circ M$ satisfies the condition of Section 2. In other words we are in position to use constants  $a$ and $z_\circ$, where $ f(z_\circ)=0$.
Using Lemma~\ref{mainlemma} (the  localization of $z_\circ=r_0e^{\imath t_\circ}$) and formulae for $a$ given by \eqref{eq:abolje}, \eqref{eq:kosargumenta} and \eqref{eq:sinargumenta}, we find:
\begin{equation}
\label{eq:skalprod}\begin{split}
\Re (z_\circ\overline{a})&= |z_\circ||a|\big(\cos t_\circ \cos\delta+\sin t_\circ\sin\delta\big)\\
\\&= -\frac{|z_\circ||a|}{\sin\frac{q}{2}\sqrt{1-\sin p\sin(q-p)}}\\ &\times\bigg(\sin(\frac{q}{2}-p)\cos t_\circ +\cos\frac{q}{2}\sin t_\circ\sqrt{\sin p\sin(q-p)}\bigg)\leq 0.\end{split}
\end{equation}
Indeed, for $q<\min\{\pi, 2p\}$ we have $t_\circ \in (\alpha, \pi)$ and, since $\cos\alpha=\frac{\tan(\frac{q}{2}-p)}{\tan\frac{q}{2}}<0$ we conclude $t_\circ \in (\frac{\pi}{2},\pi)$ and $\sin t_\circ>0, \cos t_\circ<0.$ In this case, $\sin(\frac{q}{2}-p)<0$ and $\cos\frac{q}{2}>0$, hence the inequality follows.

Similarly, for $q >\max\{\pi, 2p\}$ we have found that $t_\circ \in (\pi+\alpha, 2\pi),$ and, by $\cos\alpha=\frac{\tan(\frac{q}{2}-p)}{\tan\frac{q}{2}} <0$ it follows that $t_\circ \in (\frac{3\pi}{2},2\pi)$ and $\sin t_\circ<0, \cos t_\circ>0.$ Also, $\sin(\frac{q}{2}-p)>0$ and $\cos\frac{q}{2}<0$ and we are done again.

The same reasoning works for the two remaining cases. If $q \in (\pi, 2p)$ then $t_\circ \in (\pi, \pi+\alpha),$ while $\cos\alpha=\frac{\tan(\frac{q}{2}-p)}{\tan\frac{q}{2}}>0$ gives $\alpha \in (0,\frac{\pi}{2})$ and $t_\circ \in (\pi, \frac{3\pi}{2})$, thus $\sin t_\circ<0, \cos t_\circ<0.$ Now,
$\sin(\frac{q}{2}-p)<0$ and $\cos\frac{q}{2}<0$ implies the non-positivity of $\Re (z_\circ\overline{a}).$ If $q \in (2p,\pi)$ then $t_\circ \in (0, \alpha),$ with $\cos\alpha=\frac{\tan(\frac{q}{2}-p)}{\tan\frac{q}{2}}>0$ which leads to $t_\circ \in (0,\frac{\pi}{2})$ and $\sin t_\circ>0$ and $\cos t_\circ>0.$ The conclusion follows from $\sin(\frac{q}{2}-p)>0$ and $\cos\frac{q}{2}>0.$ By using continuity, similarly as in the proof of the previous lemma, we get the conclusions in cases $a')-d').$

From the formula \eqref{eq:krivina3} we get
$$|\mathcal{K}| \leq \frac{\pi^2}{2}\frac{|1-z_\circ^2|^2|z_\circ^2-e^{2\imath\alpha}|^2}{(1+|z_\circ|^2)^4}\leq \frac{\pi^2}{2},$$
which was to be proved.
\end{proof}
\subsection*{Acknowledgements}
We are very thankful to the referee for their careful reading of the paper and for their substantial suggestions and insightful comments, which have significantly improved the quality of our work.
The first author was supported by the Ministry of Education, Science and Innovation of Montenegro through the grant \emph{Mathematical Analysis, Optimization and Machine Learning}. Additionally, the first author received partial support from the research fund of the University of Montenegro. The second author was partially supported by the Ministry of Education, Science and Technological Development of Serbia under grant no. 174017. The authors would like to thank Franc Forstneri\v c for fruitful and inspiring conversation and his substantial observations and Antonio Ros for encouraging to work on this problem. They also wish to thank Kristian Seip, Anti Per\"al\"a, Igor Uljarevi\'c and Darko Mitrovi\'c for a number of useful remarks.

\noindent David Kalaj

\noindent University of Montenegro, Faculty of Natural Sciences and Mathematics, 81000, Podgorica, Montenegro

\noindent e-mail: {\tt davidk@ucg.ac.me}

\noindent Petar Melentijevi\'c

\noindent University of Belgrade, Faculty of Mathematics, 11000, Belgrade, Serbia,

\noindent e-mail: {\tt petarmel@matf.bg.ac.rs}

\end{document}